\documentclass[10pt]{amsart}

\usepackage{amsmath,amsthm,amssymb,amsaddr}
\usepackage{enumerate,tikz,hyperref,graphics,bussproofs,todonotes,comment}

\newcommand{\set}[2]{\{ #1 \mid #2 \}}

\newcommand{\card}[1]{| #1 |}

\newcommand{\boldtau}{\boldsymbol{\tau}}

\newcommand{\assign}{\mathrel{:=}}

\newcommand{\seq}{\vartriangleright}
\newcommand{\dmneg}{\neg}

\newcommand{\wit}{\mathord{\raisebox{\depth}{\rotatebox[origin=c]{180}{$\mathrm{w}$}}}}

\newtheorem*{theorem*}{Theorem}
\newtheorem{theorem}{Theorem}
\newtheorem{lemma}[theorem]{Lemma}

\newtheorem{definition}[theorem]{Definition}

\DeclareMathOperator{\Th}{Th}
\DeclareMathOperator{\Hen}{Hen}

\newcommand{\STH}{\mathcal{ST}^{H}}

  \newcommand{\MQST}{\mathcal{MQST}}

\begin{document}
\title{Sequent calculi for first-order $\mathrm{ST}$}

\author{Francesco Paoli}
\address{Dipartimento di Pedagogia, Psicologia, Filosofia, Universit\`a di Cagliari}

\author{Adam P\v{r}enosil}
\address{Departament de Filosofia, Universitat de Barcelona}

\maketitle
\begin{abstract}
Strict-Tolerant Logic ($\mathrm{ST}$) underpins na\"{i}ve theories of truth and vagueness (respectively including a fully disquotational truth predicate and an unrestricted tolerance principle) without jettisoning any classically valid laws. The classical sequent calculus without Cut is sometimes advocated as an appropriate proof-theoretic presentation of $\mathrm{ST}$. Unfortunately, there is only a partial correspondence between its derivability relation and the relation of local metainferential $\mathrm{ST}$-validity -- these relations coincide only upon the addition of elimination rules and only within the propositional fragment of the calculus, due to the non-invertibility of the quantifier rules. In this paper, we present two calculi for first-order $\mathrm{ST}$ with an eye to recapturing this correspondence in full. The first calculus is close in spirit to the Epsilon calculus. The other calculus includes rules for the discharge of sequent-assumptions; moreover, it is normalisable and admits interpolation. 
\end{abstract}
\section{Introduction}

Strict-Tolerant Logic ($\mathrm{ST}$) \cite{Cob1, Cob2, CEPV, Rip1, Ripley} has been at the centre of thoroughgoing debates in philosophical logic over the last decade or so. In the intentions of its propounders, a recourse to $\mathrm{ST}$ as a logical basis allows truth theorists to retain a fully disquotational truth predicate without having to forswear classical logic ($\mathrm{CL}$), and without incurring the penalty of paradoxes. In a similar way, it is possible to build on top of $\mathrm{ST}$ a theory of vagueness that includes an unrestricted tolerance principle -- again, with no obligation to give up any single classical tautology or classically valid argument schema. 

Cobreros, Egr\'e, Ripley, and van Rooij maintain that true sentences can be either \emph{strictly} or \emph{tolerantly} true. Likewise, false sentences can be either strictly or tolerantly false. An $\mathrm{ST}$-valid argument is an argument that never leads from strictly true premises to a strictly false conclusion -- it may lead, though, from strictly true premises to a tolerantly false conclusion. These informal remarks can be recast into a rigorous semantic framework by recourse to first-order $3$-valued strong Kleene models (see below), where a strictly (tolerantly) true sentence is assigned a value that is equal to $1$ (greater than $0$), and a strictly (tolerantly) false sentence is assigned a value that is equal to $0$ (smaller than $1$). Importantly, $\mathrm{ST}$-valid sequents (viewed as argument forms) are no more and no less than the classically valid sequents -- hence the contention that $\mathrm{CL}$ is not being maimed. Despite this, the transitivity of consequence is not valid without restriction (see e.g. \cite{Cob2}): it may be the case that a strictly true premiss $\varphi$ entails a tolerantly true conclusion $\psi$, which in turn entails a strictly false sentence $\chi$, while $\varphi$ fails to entail $\chi$. Due to this failure of transitivity, the addition of principles for disquotational truth to first-order $\mathrm{ST}$ does not lead to triviality. For instance, although the Liar sentence is a theorem of the resulting theory, and although any formula follows from the Liar sentence in $\mathrm{ST}$,  one cannot concatenate these derivations. Similar considerations hold for the possibility of a classically-based theory of vagueness containing a principle of tolerance for vague predicates. 

In a nutshell, Cobreros and colleagues hold that $\mathrm{ST}$ can solve the paradoxes without having, like many of its rivals, to pay the price of mutilating $\mathrm{CL}$ in the process. Whether they are actually in a position to keep these generous promises is a matter of some controversy (see e.g.\ \cite{Barrio, Barpai, DP, Pailos, Priest} for some critical views), into which we will not enter here, although some brief remarks on the issue will be reserved for the next section.

Depending on the occasion, $\mathrm{ST}$ is presented via models or via proof systems. Ripley, in particular, against the backdrop of his bilateralist (hence inferentialist) views, tends to favour a formulation of first-order $\mathrm{ST}$ in terms of a sequent calculus -- indeed, the classical sequent calculus minus the rule of Cut \cite{Ripley}. Due to Gentzen's Cut Elimination theorem, the provable sequents of this calculus are precisely the classically valid sequents, which, by the above, are none other than the $\mathrm{ST}$-valid ones. Due to the absence of Cut, moreover, one can safely supplement it with rules for, say, disquotational truth while eschewing the paradoxical derivations, because transitivity is blocked at the appropriate places.

We will see in the next section that the adequacy of this Gentzen-style rendition of $\mathrm{ST}$ is problematic. In point of fact, a version of the classical \emph{propositional} sequent calculus can yield a system that is strongly complete with respect to the above semantics, in a substantive sense that we will render precise in due course. Crucially, to attain completeness, one must augment the standard classical calculus, whose operational part consists in introduction rules only, with \emph{elimination} rules that invert the introductions. This is a feasible goal, because all the sentential connectives can be given invertible rules. The first-order calculus, though, is quite another matter. The left rule for the universal quantifier and the right rule for the existential quantifier are anything but invertible. The above strategy cannot be straightforwardly carried over to the first-order level.

In this paper, we present some suggestions to overcome this problem. After providing some basic information on the proof theory and semantics of $\mathrm{ST}$ (§ \ref{robertafadda}), we introduce two sequent calculi that are sound and complete for first-order $\mathrm{ST}$, based on different ideas. The first calculus, $\mathcal{ST^H}$, requires an expansion of the chosen first-order signature by denumerably many individual constants, the \emph{Henkin constants}, whose role is to act as \textquotedblleft witnesses\textquotedblright ~for the different existential and universal formulas expressible in the language, as in the construction of the canonical model of the Henkin-style completeness theorem for first-order logic (§ \ref{ignazioputzu}). Using this device, it is possible to formulate introduction rules for quantifiers that are both invertible and free from any eigenvariable restriction. This approach has obvious similarities with Hilbert's Epsilon calculus \cite{Zach}, which we will try to elucidate (along with the existing differences). As far as we could see, this calculus lacks interesting proof-theoretical properties. In particular, it is not clear what type of normal form would be appropriate for derivations in this calculus.

The other calculus, $\mathcal{MQST}$ (§§ \ref{marcopignotti} and \ref{antonellomura}), is metainferential in character. It is based on the idea that introduction and elimination rules can \emph{discharge} sequent-assumptions, exactly like formula-assumptions can be discharged in natural deduction calculi. This leads to a formulation of the problematic elimination rules (the left universal and the right existential ones) fashioned after Schroeder-Heister's generalised elimination rules in natural deduction \cite{SH}. In particular,  $\mathcal{MQST}$ behaves better than $\mathcal{ST^H}$: we provide a normal form theorem and derive an  interpolation theorem that, in so far as we are working in a system whose deductive strength is intermediate between the classical sequent calculus with Cut and its counterpart without Cut, can hopefully be viewed with interest also by classically-minded logicians.

\section{Strict-Tolerant Logic: Its Semantics and Proof Theory}\label{robertafadda}

As hinted above, the whole project of Strict-Tolerant Logic rests on the idea that there are two modes of truth and falsity: a \emph{strict} and a \emph{tolerant} mode. No strictly true sentence can be strictly false at the same time. However, there are sentences that are both tolerantly true and tolerantly false. Pertinent examples are paradoxical sentences like the Liar, or sentences that ascribe a vague predicate to one of its borderline cases of application.

To cash out this insight in formal terms, Cobreros and colleagues  \cite{Cob1, Cob2, CEPV, Rip1, Ripley} provide a $3$-valued semantics for the language of classical first-order logic, where the Boolean values $1$ and $0$ correspond, respectively, to strict truth and strict falsity, while tolerant truth and tolerant falsity find a home in the semantics thanks to the presence of the non-classical value $\frac{1}{2}$. Namely, a sentence is tolerantly true in a model if it is assigned therein either the value $1$ or the value $\frac{1}{2}$, and it is tolerantly false in a model if it is assigned therein either the value $\frac{1}{2}$ or the value $0$. The evaluation clauses for the connectives and the quantifiers are the familiar clauses of the \emph{strong Kleene} semantics for a first-order language, well-known in the literature on paradox for its use e.g. in Kripke's theory of truth \cite{Kripke}.

More precisely, let $\mathcal{L}$ be a signature for classical first-order logic, consisting of relation symbols and function symbols of finite (possibly zero) arity. $\mathcal{L}$-formulas are defined as usual -- quantifier-free formulas are referred to as $\mathcal{L}$-\emph{P-formulas}. $\varphi, \psi, ...$ are used as variables for $\mathcal{L}$-(P-)formulas, and $\Gamma, \Delta, ...$ as variables for sets of $\mathcal{L}$-(P-)formulas. An $\mathcal{L}$-\emph{(P)-sequent} is an ordered pair of finite sets of $\mathcal{L}$-(P)-formulas, noted $\Gamma \seq \Delta$. The set of all $\mathcal{L}$-sequents will be sometimes denoted by $Seq_\mathcal{L}$. 

An \emph{$\mathrm{ST}$-model} for $\mathcal{L}$ is a pair $\mathsf{M} = \langle D, I\rangle$ such that:
\begin{itemize}
    \item $I(P^n)\colon D^n \to \{ 0, \frac{1}{2}, 1 \}$, for any $n$-ary relation symbol $P^n$;
    \item $I(x) \in D$, for any variable $x$;
    \item $I(f^n)\colon D^n \to D$, for any $n$-ary function symbol $f^n$;
    \item for any atomic $\mathcal{L}$-formula $P^{n}(t_1,...,t_n)$,
    \[
    I(P^{n}(t_1,...,t_n)) = I(P^n)(I(t_1),...,I(t_n));
    \]
    \item for any $\mathcal{L}$-formulas $\varphi, \psi$, $I(\lnot \varphi) = 1 - I(\varphi)$, $I(\varphi \land \psi) = \min (I(\varphi), I(\psi))$ and $I(\varphi \lor \psi) = \max (I(\varphi), I(\psi))$;
    \item $I(\forall x\, \varphi(x)) = \min (I'(\varphi (x)))$, for all $x$-variants $I'$ of $I$;
    \item $I(\exists x\, \varphi(x)) = \max (I'(\varphi (x)))$, for all $x$-variants $I'$ of $I$.
\end{itemize}

So much for the semantics. The next task in the $\mathrm{ST}$-theorist's agenda is to define a concept of logical consequence. For reasons on which we will not dwell, Cobreros, Egr\'e, Ripley and van Rooj adopt a multiple-conclusion notion. The guiding intuition is clear -- a set of conclusions follows from a certain set of premisses if there is no way for the premisses to be all \emph{strictly true} while the conclusions are all \emph{strictly false}. In other words, this happens if every model where all the premisses are stricty true is a model where at least one conclusion is tolerantly true. 

Since we deal with sequent calculi in this paper, the formal definition of multiple-conclusion consequence in $\mathrm{ST}$ will be recast in terms of a definition of validity for $\mathcal{L}$-sequents. We say that an $\mathrm{ST}$-model (for $\mathcal{L}$) $\mathsf{M} = \langle D, I\rangle$ $\mathrm{ST}$-\emph{satisfies} an $\mathcal{L}$-sequent $\Gamma \seq \Delta$ (in symbols, $\mathsf{M} \models_{\mathrm{ST}} \Gamma \seq \Delta$) if either there is $\varphi \in \Gamma$ such that $I(\varphi) \in \{ 0, \frac{1}{2}\}$ or there is $\psi \in \Delta$ such that $I(\psi) \in \{ 1, \frac{1}{2}\}$. We also say that an $\mathcal{L}$-sequent $\Gamma \seq \Delta$ is $\mathrm{ST}$-\emph{valid} (in symbols, $ \models_{\mathrm{ST}} \Gamma \seq \Delta$) if for all $\mathrm{ST}$-models $\mathsf{M}$, $\mathsf{M} \models_{\mathrm{ST}} \Gamma \seq \Delta$. Denoting by $ \models_{\mathrm{CL}} \Gamma \seq \Delta$ the usual notion of validity for classical sequents, one can establish that:

\begin{lemma}[\cite{Cob1}; see also \cite{Girard}]\label{predisposto}
    $ \models_{\mathrm{ST}} \Gamma \seq \Delta$ iff $ \models_{\mathrm{CL}} \Gamma \seq \Delta$.
\end{lemma}

Thus, the $\mathrm{ST}$-valid sequents are precisely the classically valid ones. $\mathrm{ST}$ and $\mathrm{CL}$, though, do not validate the same sequent-to-sequent inferences, also called \emph{meta\-inferences}. \emph{Transitivity} is a case in point. In a theory of truth based on $\mathrm{ST}$, the Liar $\lambda$ is such that for all $\mathsf{M} = \langle D, I\rangle$ we have $I(\lambda) = \frac{1}{2}$. Then for atomic $\varphi, \psi$ we will have $\models_{\mathrm{ST}} \varphi \seq \lambda$ and $\models_{\mathrm{ST}} \lambda \seq \psi$, yet it won't be the case that $\models_{\mathrm{ST}} \varphi \seq \psi$. In $\mathrm{CL}$, on the other hand, transitivity holds without restrictions. Therefore, it has been claimed that the identification of $\mathrm{ST}$ and $\mathrm{CL}$ is questionable \cite{Barrio, Barpai, DP, Pailos, Priest}. Interestingly, there are logics that validate the same inferences \emph{and} metainferences as $\mathrm{CL}$, but differ from it at the level of inferences among metainferences. Examples like these pose many technical and philosophical questions, addressed in the flourishing literature on \emph{metainferential logics} (see \cite{Barrio,Barpai,Barriomore,Pailos}).

Capitalising on the co-extensivity of $\mathrm{ST}$-valid and classically valid sequents, Ripley \cite{Ripley} invites us to view the classical sequent calculus without Cut as a proof-theoretically adequate presentation of first-order $\mathrm{ST}$. Indeed, because of Gentzen's Cut Elimination result, the provable sequents of this calculus are precisely the classically valid sequents -- hence, by Lemma \ref{predisposto}, the $\mathrm{ST}$-valid ones. Due to the absence of Cut, moreover, one can safely supplement it with rules for, say, disquotational truth while still eschewing the paradoxical derivations.

A possible objection to the plausibility of this suggestion is that the \emph{derivability relation} of such a calculus, arguably its distinctive earmark (as opposed to the set of its theorems, which can be shared with other calculi, crucially including the fully transitive classical calculus), does not correspond to any significant relation among sequents definable in terms of the $\mathrm{ST}$ semantics. In particular, let us say that an $\mathcal{L}$-sequent $S$ $\mathrm{ST}$-\emph{follows} from a set $X$ of $\mathcal{L}$-sequents (in symbols, $X \models_{\mathrm{ST}} S$) in case for any $\mathrm{ST}$-model $\mathsf{M}$, if $\mathsf{M} \models_{\mathrm{ST}} S'$ for any $S' \in X$, then $\mathsf{M} \models_{\mathrm{ST}} S$. This relation, sometimes called \emph{local metainferential validity}, is considered by many practicioners of metainferential logics as the most appropriate notion of sequent-to-sequent consequence in the context of the above semantics (although the debate is lively: See e.g. \cite{CEPV, DST, DP, FR, French, Golan, Ripley, Teijeiro}). This relation is, however, strictly larger than the derivability relation of the classical sequent calculus without Cut.

This shortcoming is readily mended at the \emph{propositional} level, using a calculus introduced by Pynko~\cite{Pynko} which augments classical propositional sequent calculus without Cut with elimination rules for all logical connectives. Indeed, in \cite{DP} it is shown that the derivability relation in this calculus coincides with local meta\-inferential validity. This is a viable move because all the connectives can be given \emph{invertible} introduction rules, in such a way that the attendant eliminations are precisely the inverses of the respective introductions. This sort of \textquotedblleft harmony\textquotedblright~is crucial for the completeness proof: the canonical valuation would not be a valuation at all in absence of this property.

The propositional calculus $\mathcal{ST}^P$ with invertible rules is defined in Figure~\ref{fig: rules}, where $\varphi, \psi, ...$ denote $\mathcal{L}$-P-formulas and $\Gamma, \Delta, ...$ denote sets of $\mathcal{L}$-P-formulas.
 Note that all logical rules are \emph{bidirectional}, i.e., they can be applied top-down as well as bottom-up. Namely, $\mathcal{ST}^P$ contains the inverses of all the rules for the sentential connectives, where it is understood that a $2$-premise rule $\frac{S_1, S_2}{S}$ has two inverses, $\frac{S}{S_1}$ and $\frac{S}{S_2}$. By means of the bottom-up elimination rules one can restore in $\mathcal{ST}^P$ some of the classical derivations that are impeded by the absence of Cut. Hence, as anticipated:

\begin{figure}
\caption{Rules of $\mathcal{ST}^P$}
\label{fig: rules}

\medskip

\textbf{Propositional logical rules}

\begin{prooftree}
\def\fCenter{\seq}
\Axiom$\varphi, \psi, \Gamma \fCenter \Delta$\RightLabel{\scriptsize ($\land$L)}
\doubleLine
\UnaryInf$\varphi \wedge \psi, \Gamma \fCenter \Delta$

\def\fCenter{\seq}
\Axiom$\Gamma \fCenter \Delta, \varphi$
\Axiom$\Gamma \fCenter \Delta, \psi$\RightLabel{\scriptsize ($\land$R)}
\doubleLine
\BinaryInf$\Gamma \fCenter \Delta, \varphi \wedge \psi$

\noLine
\BinaryInfC{}
\end{prooftree}

\begin{prooftree}
\def\fCenter{\seq}
\Axiom$\varphi, \Gamma \fCenter \Delta$
\Axiom$\psi, \Gamma \fCenter \Delta$\RightLabel{\scriptsize ($\lor$L)}
\doubleLine
\BinaryInf$\varphi \vee \psi, \Gamma \fCenter \Delta$

\def\fCenter{\seq}
\Axiom$\Gamma \fCenter \Delta, \varphi, \psi$\RightLabel{\scriptsize ($\lor$R)}
\doubleLine
\UnaryInf$\Gamma \fCenter \Delta, \varphi \vee \psi$

\noLine
\BinaryInfC{}
\end{prooftree}

\begin{prooftree}
\def\fCenter{\seq}
\Axiom$\varphi, \Gamma \fCenter \Delta$\RightLabel{\scriptsize ($\lnot$R)}
\doubleLine
\UnaryInf$\Gamma \fCenter \Delta, \dmneg \varphi$

\def\fCenter{\seq}
\Axiom$\Gamma \fCenter \Delta, \varphi$\RightLabel{\scriptsize ($\lnot$L)}
\doubleLine
\UnaryInf$\dmneg \varphi, \Gamma \fCenter \Delta$

\noLine
\BinaryInfC{}
\end{prooftree}

\medskip

\textbf{Structural rules: Weakening}

\begin{prooftree}
\def\fCenter{\seq}
\Axiom$\Gamma \fCenter \Delta$\RightLabel{\scriptsize (WL)}
\UnaryInf$\varphi, \Gamma \fCenter \Delta$

\def\fCenter{\seq}
\Axiom$\Gamma \fCenter \Delta$\RightLabel{\scriptsize (WR)}
\UnaryInf$\Gamma \fCenter \Delta, \varphi$

\noLine
\BinaryInfC{}
\end{prooftree}

\medskip

\textbf{Structural rules: Identity}

\begin{prooftree}
\def\fCenter{}
\Axiom$\fCenter\vphantom{\Sigma}$\RightLabel{\scriptsize (ID)}
\def\fCenter{\seq}
\UnaryInf$\varphi \fCenter \varphi\vphantom{\Sigma}$

\end{prooftree}

\end{figure}

\begin{theorem}[\cite{DP}]
If $X \cup \{ S \}$ is any set of $\mathcal{L}$-P-sequents, then $X \vdash_{\mathcal{ST}^P} S$ iff $X \models_{\mathrm{ST}} S$.
\end{theorem}

Recall that the \emph{external consequence relation} of a sequent calculus $\mathcal{C}$, of signature $\mathcal{L}_\mathcal{C}$, holds between a set $\Gamma$ of $\mathcal{L}_\mathcal{C}$-formulas and an $\mathcal{L}_\mathcal{C}$-formula $\varphi$ when $\seq \varphi$ is derivable in $\mathcal{C}$ from $\{ \seq \psi \mid \psi \in \Gamma\}$ \cite{Avron}. Interestingly, this consequence relation coincides with a familiar logic:

\begin{theorem}[\cite{Pynko}, see also \cite{Barrio, Prenosil}]
The external consequence relation of $\mathcal{ST}^P$ coincides with the consequence relation of (propositional) $\mathrm{LP}$.
\end{theorem}

This approach is ineffective for the full first-order calculus, because not all the usual rules for the universal and the existential quantifiers are invertible. For future reference, however, let us introduce here an appropriate version of the sequent calculus for classical first order logic minus the rule of Cut, a slight variant of the one in \cite{Ripley}. The calculus $\mathcal{ST}^Q$ differs from $\mathcal{ST}^P$ in the following aspects:
\begin{itemize}
    \item its syntactic units are $\mathcal{L}$-sequents (possibly containing quantifiers) as opposed to $\mathcal{L}$-P-sequents;
    \item it contains introduction rules only, i.e., the bottom-up directions of the logical rules in $\mathcal{ST}^P$ are deleted;
    \item it contains the additional rules
    \begin{prooftree}
    \def\fCenter{\seq}
    \Axiom$ \varphi[x \mapsto t], \Gamma  \fCenter \Delta$\RightLabel{\scriptsize ($\forall$L)}
    \UnaryInf$\forall x\,\varphi(x), \Gamma \fCenter \Delta$

    \def\fCenter{\seq}
    \Axiom$\Gamma \fCenter \Delta, \varphi[x \mapsto t]$\RightLabel{\scriptsize ($\forall$R)}
    \UnaryInf$\Gamma \fCenter \Delta, \forall x\,\varphi(x)$
    
    \noLine
    \BinaryInfC{}
    \end{prooftree}
    (and dual rules for the existential quantifier), where in both rules the notation $\varphi[x \mapsto t]$ denotes the result of substituting $t$ for $x$ in $\varphi$.
   This notation assumes that the substitution does not result in the capture of any variables which occur in $t$; in the right introduction rule, moreover, $x$ does not occur in $\Gamma \cup \Delta$ and thus the eigenvariable condition has to be respected.
\end{itemize}

 Due to the non-invertibility of the left introduction for the universal quantifier, and dually, of the right introduction for the existential quantifier, $\mathcal{ST}^Q$ cannot be supplemented with elimination rules and so its derivability relation does not match local metainferential validity.

In the rest of this paper, we will aim at upgrading $\mathcal{ST}^Q$ to a strongly complete calculus for $\mathrm{ST}$.

\section{The Calculus $\STH$}\label{ignazioputzu}

For a start, let us bring into sharper focus the non-invertibility issue for the quantifier rules in $\mathcal{ST}^Q$. By way of example, consider the left rule for the universal quantifier:
\begin{prooftree}
    \def\fCenter{\seq}
    \Axiom$ \varphi[x \mapsto t], \Gamma  \fCenter \Delta$\RightLabel{\scriptsize ($\forall$L)}
    \UnaryInf$\forall x\,\varphi(x), \Gamma \fCenter \Delta$
    \end{prooftree}
Were the rule invertible, this would license the inference of, say, $\psi$ from $\varphi[x \mapsto t]$ on the assumption that $\psi$ is inferrable from $\forall x\,\varphi(x)$ -- which doesn't work in general, as we are trading a stronger hypothesis for a weaker one. Things would be different if in the upper sequent, instead of a generic term $t$, we substituted for the free variable $x$ a term that \emph{witnesses the availability of the universal sentence $\forall x\,\varphi(x)$}. Thus, a way to get over the hump is to expand $\mathcal{L}$ to a language where there are enough such terms -- i.e., where all universal and existential sentences are witnessed, like in the Henkin-style construction of the canonical model for classical first-order logic. We will explore this avenue in the present section.

\subsection{Presentation of the calculus}

  Consider a first-order signature $\mathcal{L}$
  which consists of relation symbols and function symbols of finite (possibly zero) arity. We assume that the signature contains at least one relation symbol (otherwise the set of formulas is empty). The \emph{immediate Henkin expansion} of $\mathcal{L}$ is the signature that expands $\mathcal{L}$ by a new constant $\wit(\forall x\,\varphi)$ for each universal $\mathcal{L}$-formula\footnote{For the benefit of readers familiar with Henkin's completeness proof for first-order logic, we emphasize that $\wit(\forall x\,\varphi)$ and $\wit(\exists x\,\varphi)$ are added even for formulas $\varphi$ which contain free variables other than $x$.} $\forall x\,\varphi$ (called a \emph{universal Henkin constant}), and by a new constant $\wit(\exists x\,\varphi)$ for each existential $\mathcal{L}$-formula $\exists x\,\varphi$ (called an \emph{existential Henkin constant}). If $\mathcal{L}_{i}$ for $i \in \omega$ is a sequence of signatures such that $\mathcal{L}_{0} = \mathcal{L}$ and $\mathcal{L}_{i+1}$ is the immediate Henkin expansion of $\mathcal{L}_{i}$, we call $\Hen \mathcal{L} \assign \bigcup_{i \in \omega} \mathcal{L}_{i}$ the \emph{Henkin expansion} of $\mathcal{L}$. Terms and formulas of $\Hen \mathcal{L}$ will be called \emph{$\mathcal{L}$-Henkin} terms and formulas. The interpretation of the Henkin constants is subject to the following constraints.
  
\begin{definition}
Let $\mathsf{M} = \langle D, I\rangle$ be an $\mathrm{ST}$-model for $\mathcal{L}$. A \emph{Henkin expansion} of $\mathsf{M}$ is an $\mathrm{ST}$-model $\mathsf{M}^H = \langle D, I^H \rangle$ for $\Hen \mathcal{L}$, where $I^H(\varphi) = I(\varphi)$ for each $\mathcal{L}$-formula $\varphi$, and moreover for each $\mathcal{L}$-Henkin formula $\psi$:
\begin{align*}
  & I(\psi[x \mapsto \wit (\forall x\, \psi)]) = I(\forall x\, \psi), & & I(\psi[x \mapsto \wit (\exists x\, \psi)]) = I(\exists x\, \psi).
\end{align*}
An \emph{$\mathcal{L}$-Henkin model} is an $\mathrm{ST}$-model for $\Hen \mathcal{L}$ which is the Henkin expansion of some $\mathrm{ST}$-model for $\mathcal{L}$.
\end{definition}
  Since formulas are interpreted in an $\mathrm{ST}$-model by elements of a finite linearly ordered algebra,
 every $\mathrm{ST}$-model has a Henkin expansion. In general, this expansion need not be unique.

  Sequents in the signature $\Hen \mathcal{L}$ will be called \emph{$\mathcal{L}$-Henkin sequents}.

\begin{definition}
Let $\Gamma \seq \Delta$ be an $\mathcal{L}$-Henkin sequent, where we fix enumerations $ \langle \gamma_1,...,\gamma_n \rangle$ of $\Gamma$ and $ \langle \delta_1,...,\delta_m \rangle$ of $\Delta $. The \emph{formula translation} $\boldtau(\Gamma \seq \Delta)$ of $\Gamma \seq \Delta$ is defined as follows (all disjunctions are associated to the left):
\begin{itemize}
    \item $\lnot\gamma_1 \lor... \lor \lnot \gamma_n \lor \delta_1 \lor... \lor \delta_m$, if $n,m \geq 1$;
    \item $\lnot\gamma_1 \lor... \lor \lnot \gamma_n$, if  $n \geq 1, m = 0$;
    \item $\delta_1 \lor ... \lor \delta_m$, if $n = 0, m \geq 1$;
    \item $\varphi_0 \land \lnot \varphi_0$, where $\varphi_0$ is a fixed atomic $\mathcal{L}$-formula, if $n,m = 0$.
\end{itemize}
\end{definition}
 
Throughout the following, sets of $\mathcal{L}$-Henkin sequents will be denoted by $X$, $Y,...$ and individual $\mathcal{L}$-Henkin sequents by $S, S'...$ It is not hard to prove that:

\begin{lemma}\label{claudioternullo}
  The sequents $S$ and $\emptyset \seq \boldtau(S)$ are interderivable in $\STH$.
\end{lemma}

\begin{definition}
  An $\mathcal{L}$-Henkin sequent $S$ is an \emph{$\mathrm{ST}^H$-consequence} of a set $X$ of $\mathcal{L}$-Henkin sequents, in symbols $X \models_{\mathrm{ST}^H} S$, in case for each $\mathcal{L}$-Henkin model $\mathsf{M}^{H}$, if $\mathsf{M}^{H} \models_{\mathrm{ST}} S'$ for all $S' \in X$, then $\mathsf{M}^{H} \models_{\mathrm{ST}} S$.
\end{definition}

  We capture this consequence relation between $\mathcal{L}$-Henkin sequents by means of the calculus $\STH$, which extends the calculus $\mathcal{ST}^{P}$ by the rules listed in Figure~\ref{fig: rules1}. Observe that these rules do not require any eigenvariable restriction. We merely require, as usual, that in $\varphi[x \mapsto t]$ the term $t$ can be substituted for $x$. We abbreviate the claim that $S$ is derivable from $X$ in the calculus $\STH$ by $X \vdash_{\STH} S$, and in the names of rules we distinguish the direction (top-down \emph{vs} bottom-up) in which logical rules are applied by means of arrows pointing downwards or upwards.

  To give a flavour of how this calculus works, we choose to present the proof of a first-order logical truth that is often brought as an example in the context of calculi that, like the Epsilon calculus, in general allow more efficient proofs than standard sequent proofs (see e.g.\ \cite{AB, Powell} for a more precise description of this speed-up). Here is a proof of $\seq \exists x (P(x) \rightarrow \forall x\, P(x))$ in $\mathcal{ST^Q}$, where $\varphi \to \psi$ is shorthand for $\lnot \varphi \lor \psi$:
\begin{prooftree}
\def\fCenter{\seq}
\Axiom$P(x) \fCenter P(x)$
\UnaryInf$P(y), P(x) \fCenter P(x), \forall x\, P(x)$
\UnaryInf$P(y) \fCenter P(x), P(x) \rightarrow \forall x\, P(x)$
\UnaryInf$P(y) \fCenter P(x), \exists x\, (P(x) \rightarrow \forall x\, P(x))$
\UnaryInf$P(y) \fCenter \forall x\, P(x), \exists x (P(x) \rightarrow \forall x\, P(x))$
\UnaryInf$\fCenter P(y) \rightarrow \forall x\, P(x), \exists x (P(x) \rightarrow \forall x P(x))$
\UnaryInf$\fCenter \exists x(P(x) \rightarrow \forall x\, P(x)), \exists x(P(x) \rightarrow \forall x P(x))$
\UnaryInf$\fCenter \exists x(P(x) \rightarrow \forall x\, P(x))$
\end{prooftree}
Compare this to a proof of the same sequent in $\STH$, where the third inference is an application of (EWI):
\begin{prooftree}
\def\fCenter{\seq}
\Axiom$P(\wit(\forall x\, P(x))) \fCenter P(\wit(\forall x\, P(x)))$
\UnaryInf$P(\wit(\forall x\, P(x))) \fCenter \forall x\, P(x)$
\UnaryInf$\fCenter P(\wit(\forall x\, P(x))) \rightarrow \forall x\, P(x)$
\UnaryInf$\fCenter P(\wit(\exists x\, (P(x) \rightarrow \forall x\, P(x)))) \rightarrow \forall x\, P(x)$
\UnaryInf$\fCenter \exists x (P(x) \rightarrow \forall x\, P(x))$
\end{prooftree}

\subsection{Soundness and completeness}
  
It is not difficult to show, by induction on the length of derivations, that our calculus is sound.

\begin{figure}
\caption{Additional rules for $\STH$}
\label{fig: rules1}

\medskip

\textbf{Witness introduction rules}

\begin{prooftree}
\def\fCenter{\seq}
\Axiom$\varphi[x \mapsto t], \Gamma \fCenter \Delta$\RightLabel{\scriptsize (UWI)}
\UnaryInf$\varphi[x \mapsto \wit (\forall x\,\varphi)], \Gamma \fCenter \Delta$

\def\fCenter{\seq}
\Axiom$\Gamma \fCenter \Delta, \varphi[x \mapsto  t]$\RightLabel{\scriptsize (EWI)}
\UnaryInf$\Gamma \fCenter \Delta, \varphi[x \mapsto \wit (\exists x\, \varphi)]$

\noLine
\BinaryInfC{}
\end{prooftree}

\medskip

\textbf{Witness elimination rules}

\begin{prooftree}
\def\fCenter{\seq}
\Axiom$\varphi[x \mapsto \wit (\exists x\,\varphi)], \Gamma \fCenter \Delta$\RightLabel{\scriptsize (EWE)}
\UnaryInf$\varphi[x \mapsto t], \Gamma \fCenter \Delta$

\def\fCenter{\seq}
\Axiom$\Gamma \fCenter \Delta, \varphi[x \mapsto \wit (\forall x\, \varphi)]$\RightLabel{\scriptsize (UWE)}
\UnaryInf$\Gamma \fCenter \Delta, \varphi[x \mapsto t]$

\noLine
\BinaryInfC{}
\end{prooftree}

\medskip

\textbf{Logical rules for quantifiers}

\begin{prooftree}
\def\fCenter{\seq}
\Axiom$\varphi[x \mapsto \wit (\forall x\,\varphi)], \Gamma \fCenter \Delta$\RightLabel{\scriptsize ($\forall$LW)}
\doubleLine
\UnaryInf$\forall x\, \varphi, \Gamma \fCenter \Delta$

\def\fCenter{\seq}
\Axiom$\Gamma \fCenter \Delta, \varphi[x \mapsto \wit (\forall x\, \varphi)]$\RightLabel{\scriptsize ($\forall$RW)}
\doubleLine
\UnaryInf$\Gamma \fCenter \Delta, \forall x\, \varphi$

\noLine
\BinaryInfC{}
\end{prooftree}

\begin{prooftree}
\def\fCenter{\seq}
\Axiom$\varphi[x \mapsto \wit (\exists x\,\varphi)], \Gamma \fCenter \Delta$\RightLabel{\scriptsize ($\exists$LW)}
\doubleLine
\UnaryInf$\exists x\, \varphi, \Gamma \fCenter \Delta$

\def\fCenter{\seq}
\Axiom$\Gamma \fCenter \Delta, \varphi[x \mapsto \wit (\exists x\, \varphi)]$\RightLabel{\scriptsize ($\exists$RW)}
\doubleLine
\UnaryInf$\Gamma \fCenter \Delta, \exists x\, \varphi$

\noLine
\BinaryInfC{}
\end{prooftree}

\medskip

\end{figure}

\begin{theorem}
  Let $X \cup \{ S \}$ be a set of $\mathcal{L}$-sequents. If $X \vdash_{\mathcal{ST}^H} S$, then $X \models_{\mathrm{ST}^H} S$.
\end{theorem}

The converse direction requires the construction of a canonical model along the lines of the completeness proof for $\mathcal{ST}^P$ given in \cite{DP}. The cumbersome construction employed on that occasion will be considerably streamlined, though. In order to do so, let us first recapitulate some elementary concepts and facts from lattice theory. 

Recall that a lattice $\mathbf{L}$ is \emph{complete} if meets and joins of arbitrary subsets of $L$ exist in $\mathbf{L}$. If $\mathbf{L}$ is a complete lattice, then $a \in L$ is said to be {\em compact} if for any subset $S\subseteq L$ satisfying $a\le\bigvee S$, there exists a finite subset $S'\subseteq S$ such that $a\le\bigvee S'$. A complete lattice $\mathbf{L}$ is {\em algebraic} if each of its elements is a join of compact elements.

An element $a$ of a lattice $\mathbf{L}$ is said to be {\em meet-irreducible} if $\bigwedge X=a$,
 where $X$ is a finite subset of $L$,
implies $a\in X$, and {\em meet-prime} if $\bigwedge X\le a$,
  where $X$ is a finite subset of $L$
implies $x\le a$, for some $x\in X$. It is easily proved that every meet-prime element is meet-irreducible, and that if $\mathbf{L}$ is distributive, the converse also holds.

An element $a$ of a complete lattice $\mathbf{L}$ is said to be {\em completely meet-irreducible} if $\bigwedge X=a$,
 where $X \subseteq L$, implies $a\in X$, and {\em completely meet-prime} if $\bigwedge X\le a$,
 where $X \subseteq L$, implies $x\le a$, for some $x\in X$. Again, it is easily proved that every (completely) meet-prime element is (completely) meet-irreducible, and that if binary joins distribute over binary (arbitrary) meets in $\mathbf{L}$, then
 the converse also holds.

Algebraic lattices are especially well-behaved from many viewpoints. In particular, we will use below the following result, which can be established via a standard application of Zorn's lemma:

\begin{lemma}
\label{meet irr}\cite{Birkhoff} Let $\mathbf{L}$ be an algebraic lattice.
Furthermore, let $a\in L$ and let $c$ be a compact element of $L$ such that
$c\nleq a$. Then there is a completely meet-irreducible element $b\in L$ such
that $a\leq b$ and $c\nleq b$.
\end{lemma}

Now for the construction of the canonical model. We start by adapting some definitions from \cite[Def.~18.15]{DP}.

\begin{definition}
An \emph{$\mathcal{L}$-Henkin theory} is a set $T$ of $\mathcal{L}$-Henkin sequents which
contains all the provable sequents of $\STH$ and is
closed with respect to all the rules of $\STH$. 
\end{definition}
A set of $\mathcal{L}$-Henkin sequents $T$ is thus an $\mathcal{L}$-Henkin theory if and only if it is deductively closed, i.e., $T\vdash_{\mathcal{ST}^H} S$ implies $S\in T$. The $\mathcal{L}$-Henkin theory generated by $X$ will be denoted by $\Th(X)$.
The set $\mathfrak{T}$ of all $\mathcal{L}$-Henkin theories is the universe of an algebraic lattice whose compact elements have the form $\Th(X)$, for $X$ a finite set of $\mathcal{L}$-Henkin sequents. When this is not prejudicial to comprehension, we abbreviate ``$\mathcal{L}$-Henkin theory'' by ``theory''.
\begin{definition}\label{cataratta}
A theory
$T$ is:

\begin{itemize}
\item $S$\emph{-consistent}, if $S\notin T$;

\item \emph{prime}, if $\Gamma \seq \Delta\in T$ implies that either $\gamma \seq \emptyset \in T$ for some $\gamma \in \Gamma$ or $\emptyset \seq \delta \in T$ for some $\delta \in \Delta$, unless $\Gamma \cup \Delta = \emptyset$;

\item \emph{complete}, if for all formulas $\varphi$, either $\emptyset \seq \varphi\in T$ or $\varphi\seq \emptyset \in T$;

\item \emph{consistent}, if for all formulas $\varphi$, not both $\emptyset \seq
\varphi\in T$ and $\varphi\seq \emptyset \in T$.
\end{itemize}
\end{definition}

  Observe that every prime theory $T$ is complete, because $\varphi\seq\varphi\in T$ for each~$\varphi$. It will be convenient to introduce the following notation:
\begin{align*}
  & (\Gamma_{1} \seq \Delta_{1}) \sqcup (\Gamma_{2} \seq \Delta_{2}) \assign \Gamma_{1}, \Gamma_{2} \seq \Delta_{1}, \Delta_{2}, & & X \sqcup S' \assign \set{S \sqcup S'}{S \in X}.
\end{align*}

\begin{lemma}\label{Mario}
  Let $T$ be a theory. The following are equivalent:
  \begin{enumerate}
      \item $T$ is prime;
      \item $S_{1} \sqcup S_{2} \in T$ implies that either $S_{1} \in T$ or $S_{2} \in T$;
      \item $\emptyset \seq \varphi \vee \psi \in T$ implies that either $\emptyset \seq \varphi$ or $\emptyset \seq \psi$.
  \end{enumerate}
\end{lemma}

\begin{proof}
  (1) implies (2). Let $S_{1} = \Gamma \seq \Delta$ and $S_{2} = \Pi \seq \Sigma$. If $S_{1} \sqcup S_{2} \in T$, then because $T$ is prime either $\gamma \seq \emptyset \in T$ or $\emptyset \seq \delta$ or $\pi \seq \emptyset$ or $\emptyset \seq \sigma$ for some $\gamma \in \Gamma$, $\delta \in \Delta$, $\pi \in \Pi$, or $\sigma \in \Sigma$. Therefore by (WL), (WR) either $\Gamma \seq \Delta \in T$ (in the first two cases) or $\Pi \seq \Sigma \in T$ (in the last two cases).

  (2) implies (3). Let $\emptyset \seq \varphi \lor \psi \in T$. Then $\emptyset \seq \varphi,\psi \in T$. But $(\emptyset \seq \varphi, \psi) = (\emptyset \seq \varphi) \sqcup (\emptyset \seq \psi)$, so either $\emptyset \seq \varphi \in T$ or $\emptyset \seq \psi \in T$.

  (3) implies (1): Suppose that $\Gamma \seq \Delta\in T$, with $\Gamma \cup \Delta \neq \emptyset$. By Lemma \ref{claudioternullo}, $\emptyset \seq \tau(\Gamma \seq \Delta) \in T$. Since (3) holds, either there is $\gamma \in \Gamma$ such that $\emptyset \seq \lnot \gamma \in T$, or there is $\delta \in \Delta$ such that $\emptyset \seq \delta \in T$. Applying ($\lnot$L), ($\lnot$R) if necessary, we obtain that $T$ is prime.
\end{proof}

\begin{lemma}\label{lemma: pcp}
  $\Th(X, S_{1} \sqcup S_{2}) = \Th(X, S_{1}) \cap \Th(X, S_{2})$. More explicitly,
  \begin{align*}
      X, S_{1} \sqcup S_{2} \vdash_{\STH} S \iff X, S_{1} \vdash_{\STH} S \text{ and } X, S_{2} \vdash_{\STH} S.
  \end{align*}
\end{lemma}

\begin{proof}
  The left-to-right implication holds because $S_{i} \vdash S_{1} \sqcup S_{2}$ for $i \in \{ 1, 2 \}$. Conversely, consider derivations $\mathcal{D}_{1}$ and $\mathcal{D}_{2}$ witnessing that $X, S_{1} \vdash_{\STH} S$ and $X, S_{2} \vdash_{\STH} S$. Changing each sequent of $\mathcal{D}_{1}$ from $S'$ to $S' \sqcup S_{2}$ yields a derivation $\mathcal{D}_{1}'$ witnessing that $X \sqcup S_{2}, S_{1} \sqcup S_{2} \vdash_{\STH} S \sqcup S_{2}$. Changing each sequent of $\mathcal{D}_{2}$ from $S'$ to $S \sqcup S'$ yields a derivation $\mathcal{D}_{2}'$ witnessing that $S \sqcup X, S \sqcup S_{2} \vdash_{\STH} S \sqcup S$. Concatenating $\mathcal{D}_{1}'$ and $\mathcal{D}_{2}'$ yields a derivation witnessing that $S \sqcup X, X \sqcup S_{2}, S_{1} \sqcup S_{2} \vdash_{\STH} S \sqcup S$. Since $X \vdash_{\STH} S \sqcup X$ and $X \vdash_{\STH} X \sqcup S_{2}$, while $S \sqcup S$ is $S$, it follows that $X, S_{1} \sqcup S_{2} \vdash_{\STH} S$.
\end{proof}

\begin{lemma}
\label{primlemma}Let $T$ be a theory. The following are equivalent:
\begin{enumerate}
    \item $T$ is prime;
    \item $T$ is a meet-prime element of $\mathfrak{T}$;
    \item $T$ is a meet-irreducible element of $\mathfrak{T}$.
\end{enumerate}
\end{lemma}

\begin{proof}
(1) implies (2). Suppose $T$ is prime, and suppose ex absurdo that $T_{1}\cap T_{2}\subseteq
T$ but $T_{1}\nsubseteq T$ and $T_{2}\nsubseteq T$. So there exist $S_{1}\in T_{1}-T$ and $S_{2}\in T_{2}-T$, hence $S_{1} \sqcup S_{2} \in T_{1} \cap T_{2} \subseteq T$. But because $T$ is prime, by Lemma \ref{Mario} either $S_{1} \in T$ or $S_{2} \in T$, contradicting $S_{1} \in T_{1} - T$ and $S_{2} \in T_{2} - T$.

(2) implies (3) trivially. (3) implies (1). We use Lemma \ref{Mario} again. Suppose that $T$ is a meet-irreducible theory and consider $S_{1} \sqcup S_{2} \in T$. Then $T = \Th(T, S_{1} \sqcup S_{2}) = \Th(T, S_{1}) \cap \Th(T, S_{2})$ by Lemma~\ref{lemma: pcp}, hence either $T = \Th(T, S_{1})$ and $S_{1} \in T$ or $T = \Th(T, S_{2})$ and $S_{2} \in T$.
\end{proof}

The following is a simplified proof of Corollary 18.23 in \cite{DP}.

\begin{theorem}\label{pappagorgia}
  If $T$ is an $S$-consistent
theory, then there exists a prime and $S$-consistent theory $T^{\prime}$ such
that $T\subseteq T^{\prime}$.
\end{theorem}

\begin{proof}
We apply Lemma \ref{meet irr} to the lattice $\mathfrak{T}$ of all $\mathcal{L}$-Henkin theories. Let $T$ be an
$S$-consistent theory. Then $\Th\left(  S\right)  $ is a compact element of
$\mathfrak{T}$ such that $\Th\left(  S\right)  \nsubseteq T$. By Lemma
\ref{meet irr}, there exists a meet-irreducible and $S$-consistent $T^{\prime
}$ such that $T\subseteq T^{\prime}$. By Lemma \ref{primlemma},
$T^{\prime}$ is a prime theory.
\end{proof}

  The last ingredient of our completeness results is a technical lemma giving us the liberty to rename variables whenever it is convenient to do so throughout the proof.

\begin{lemma} \label{lemma: alpha equivalence}
  Suppose that $S'$ is obtained from $S$ by renaming some bound variables in some formulas. Then $S$ and $S'$ are interderivable in $\STH$.
\end{lemma}

\begin{proof}
Induction on the complexity of the formula $\varphi$ where the bound variables are renamed. We exemplify the inductive step supposing that $\varphi$ has the form $\exists x\, \psi(x)$, and that $y$ is a variable not occurring anywhere in $ S := \Gamma \seq \Delta, \exists x\, \psi(x)$:
\begin{prooftree}
\AxiomC{$\Gamma \seq \Delta, \exists x\, \psi(x)$}
\UnaryInfC{$\Gamma \seq \Delta, \psi[x \mapsto y \mapsto \wit (\exists x\, \psi (x))]$}
\UnaryInfC{$\Gamma \seq \Delta, \psi[x \mapsto y \mapsto \wit (\exists y\, \psi [x \mapsto y])]$}
\UnaryInfC{$\Gamma \seq \Delta, \exists y\, \psi[x \mapsto y]$}
\end{prooftree}
The previous inferential steps are justified as follows. The uppermost one is an application of ($\exists$RW$\uparrow$), which yields $\psi[x \mapsto y \mapsto \wit (\exists x\, \psi (x))]$ as a principal formula because $y$ does not occur in $\psi$. Next, we have an application of (EWI). Finally, we apply ($\exists$RW$\downarrow$).

If the existentially quantified formula appears on the left-hand side, we proceed similarly.
\end{proof}

\begin{theorem} \label{thm: completeness}
  Let $X \cup \{ S \}$ be a set of $\mathcal{L}$-Henkin sequents. If $X \models_{\STH} S$, then $X \vdash_{\STH} S$.
\end{theorem}

\begin{proof}
We proceed contrapositively. Suppose that $X \nvdash_{\mathrm{ST}^H} S$. Then $\Th(X)$ is $S$-consistent, and by Theorem \ref{pappagorgia} there exists a prime and $S$-consistent theory $T$ such that $\Th(X)\subseteq T$. We want to construct a canonical Henkin model $\mathsf{M}= \langle D, I\rangle$ such that $\mathsf{M}\models_{\mathrm{ST}} S'$ for all $S' \in T$, but it is not the case that $\mathsf{M}\models_{\mathrm{ST}} S$. Let thus:
\begin{itemize}
    \item $D$ be the set of all $\mathcal{L}$-Henkin terms;
    \item for any variable $x$, $I(x) = x$;
    \item for any $n$-ary function symbol $f^n$, $I(f^n)$ is the map
    \begin{align*}
    I(f^{n})\colon \langle t_{1}, \dots, t_{n} \rangle \mapsto f^n (t_{1}, \dots, t_{n});
    \end{align*} 
    \item for any $n$-ary relation symbol $P^n$,
    \[
    I(P^n)(t_{1}, \dots, t_{n}) =\begin{cases}
    1  & \text{ if } \emptyset \seq P^n(t_1,...,t_n) \in T, P^n(t_1,...,t_n) \seq \emptyset \notin T; \\
    \frac{1}{2} & \text{ if } \emptyset \seq P^n(t_1,...,t_n) \in T, P^n(t_1,...,t_n) \seq \emptyset \in T; \\
    0  & \text{ if } \emptyset \seq P^n(t_1,...,t_n) \notin T, P^n(t_1,...,t_n) \seq \emptyset \in T.
    \end{cases}
    \]
\end{itemize}

$I(P^n)$ is well-defined because $T$ is prime. Also, observe that I($\wit (\forall x\, \varphi)) = \wit (\forall x\, \varphi)$ and $I(\wit (\exists x\, \varphi)) = \wit (\exists x\, \varphi)$. More generally, $I(t) = t$ for any $\Hen \mathcal{L}$-term $t$. Hence, for atomic formulas,
\[
    I(P^n(t_{1}, \dots, t_{n})) =\begin{cases}
    1  & \text{ if } \emptyset \seq P^n(t_1,...,t_n) \in T, P^n(t_1,...,t_n) \seq \emptyset \notin T; \\
    \frac{1}{2} & \text{ if } \emptyset \seq P^n(t_1,...,t_n) \in T, P^n(t_1,...,t_n) \seq \emptyset \in T; \\
    0  & \text{ if } \emptyset \seq P^n(t_1,...,t_n) \notin T, P^n(t_1,...,t_n) \seq \emptyset \in T.
    \end{cases}
    \]

Now we prove, by induction on the complexity of $\varphi$, that the same condition holds for any formula $\varphi$:
\[
    I(\varphi) =\begin{cases}
    1  & \text{ if } \emptyset \seq \varphi \in T, \varphi \seq \emptyset \notin T; \\
    \frac{1}{2} & \text{ if } \emptyset \seq \varphi \in T, \varphi \seq \emptyset \in T; \\
    0  & \text{ if } \emptyset \seq \varphi \notin T, \varphi \seq \emptyset \in T.
    \end{cases}
\]
 The cases of propositional connectives are handled as in \cite[Thm. 18.24]{DP}. Let us consider the universal quantifier. (The case of the existential quantifier is treated similarly.)
 
 From left to right, suppose that $I(\forall x\,\varphi (x)) = 1$, whence for any $x$-variant $I'$ of $I$, $I'(\varphi (x)) = 1$. This holds in particular when $I'(x) = \wit (\forall x\,\varphi)$. By the inductive hypothesis, $\emptyset \seq \varphi[x \mapsto \wit (\forall x\,\varphi)] \in T$ and $\varphi[x \mapsto \wit (\forall x\,\varphi)] \seq \emptyset \notin T$. By ($\forall$RW$\downarrow$), $\emptyset \seq \forall x\,\varphi (x) \in T$. Suppose ex absurdo that $ \forall x\,\varphi (x) \seq \emptyset \in T$. Then by ($\forall$LW$\uparrow$), $\varphi[x \mapsto \wit (\forall x\,\varphi)] \seq \emptyset \in T$, a contradiction.
 
 Let $I(\forall x\,\varphi (x)) = 0$, whence for some $x$-variant $I'$ of $I$, $I'(\varphi (x)) = 0$. Let $t = I'(x)$. Because $t$ might not be substitutable for $x$ in $\varphi$, let $\psi$ be a formula obtained from $\varphi$ by renaming all bound variables to variables which occur in neither $\varphi$ nor $t$. Then $I'(\psi (x)) = I'(\varphi(x)) = 0$. Because $\psi$ has the same complexity as $\varphi$, by the inductive hypothesis $\emptyset \seq \psi[x \mapsto t] \notin T$ and $\psi[x \mapsto t] \seq \emptyset \in T$. By (UWI), $\psi[x \mapsto \wit (\forall x\,\psi)] \seq \emptyset \in T$, whence by ($\forall$LW$\downarrow$), $ \forall x\,\psi (x) \seq \emptyset \in T$. By Lemma~\ref{lemma: alpha equivalence} also $\forall x\,\varphi(x) \seq \emptyset \in T$. Suppose ex absurdo that $ \emptyset \seq \forall x\,\varphi (x) \in T$. Then $\emptyset \seq \forall x\,\psi(x) \in T$, so by ($\forall$RW$\uparrow$) $\emptyset \seq \varphi[x \mapsto \wit (\forall x\,\varphi)] \in T$ and by (UWE), $\emptyset \seq \varphi[x \mapsto t] \in T$, a contradiction.
 
 Let $I(\forall x\,\varphi (x)) = \frac{1}{2}$, whence for no $x$-variant $I'$ of $I$, $I'(\varphi (x)) = 0$, and for some $x$-variant $I''$ of $I$, $I''(\varphi (x)) = \frac{1}{2}$. Let $t = I''(x)$. Again, let $\psi$ be a formula obtained from $\varphi$ by renaming all bound variables to variables which occur in neither $\varphi$ nor $t$. By inductive hypothesis, $\emptyset \seq \psi[x \mapsto t] \in T$ and $\psi[x \mapsto t] \seq \emptyset \in T$. Reasoning as above, $ \forall x\,\psi (x) \seq \emptyset \in T$ and $\forall x\,\varphi(x) \seq \emptyset \in T$. Were it the case that $ \emptyset \seq \forall x\,\varphi (x) \notin T$, then also $\emptyset \seq \forall x\,\psi(x) \notin T$, so by ($\forall$RW$\downarrow$) we would have that $\emptyset \seq \psi[x \mapsto \wit (\forall x\,\psi)] \notin T$, and for $I'(x) = \wit (\forall x\,\psi)$ we would have $I'(\psi (x)) = 0$, a contradiction.
 
 From right to left, suppose first that $ \emptyset \seq \forall x\,\varphi (x) \in T$ and that $  \forall x\,\varphi (x) \seq \emptyset \notin T$. By ($\forall$RW$\uparrow$), $\emptyset \seq \varphi[x \mapsto \wit (\forall x\,\varphi)] \in T$, whence by (UWE), for any $t$ we have that $\emptyset \seq \varphi[x \mapsto t] \in T$. By ($\forall$LW$\downarrow$), $\varphi[x \mapsto \wit (\forall x\,\varphi)] \seq \emptyset \notin T$, and by (UWI), for any $t$ we have $\varphi[x \mapsto t] \seq \emptyset \notin T$. By induction this means that for any $x$-variants $I'$ of $I$ we have that $I'(\varphi (x) ) = 1$, which means $I(\forall x\,\varphi (x)) = 1$.
 
 If $ \emptyset \seq \forall x\,\varphi (x) \notin T$ and $  \forall x\,\varphi (x) \seq \emptyset \in T$, by ($\forall$RW$\downarrow$) $\emptyset \seq \varphi[x \mapsto \wit (\forall x\,\varphi)] \notin T$. So, if $I'$ is such that $I'(x) = \wit (\forall x\,\varphi)$, by induction there exists an $x$-variant of $I$ s.t. $I'(\varphi (x) ) = 0$, i.e. $I(\forall x\,\varphi (x)) = 0$.
 
 Finally, if $ \emptyset \seq \forall x\,\varphi (x) \in T$ and $  \forall x\,\varphi (x) \seq \emptyset \in T$, by ($\forall$RW$\uparrow$) $\emptyset \seq \varphi[x \mapsto \wit (\forall x\,\varphi)] \in T$, and by (UWE), $\emptyset \seq \varphi[x \mapsto t] \in T$. By induction, for all $x$-variants $I'$ of $I$ we have that $I'(\varphi (x) ) \geq \frac{1}{2}$. However, since $  \forall x\,\varphi (x) \seq \emptyset \in T$, by ($\forall$LW$\uparrow$) $\varphi[x \mapsto \wit (\forall x\,\varphi)] \seq \emptyset \in T$, hence there is an $x$-variant $I''$ of $I$ such that $I''(\varphi (x) ) = \frac{1}{2}$. So $I(\forall x\, \varphi (x)) = \frac{1}{2}$.
 
 Having established our claim, it readily follows that our Henkin model $\mathsf{M}$ does not $\mathrm{ST}$-satisfy $S = \Gamma \seq \Delta$: if it did, there would be either $\gamma \in \Gamma $ such that $\gamma \seq \emptyset \in T$ or $\delta \in \Delta $ such that $\emptyset \seq \Delta \in T$; in both cases, $S$ would belong to $T$, a contradiction. Since $T$ is a prime theory containing $X$, $\mathsf{M}$ $\mathrm{ST}$-satisfies all members of $X$, and our theorem is proved.  
\end{proof}

\subsection{Relationship with the Epsilon calculus}\label{immapinto}

In this subsection we aim at making explicit the relationship between $\STH$ and the \emph{Epsilon calculus}, already mentioned in our introduction. More precisely, we will present a sequent version $\mathcal{E}$ of the Epsilon calculus and provide two mutually inverse translations from $\STH$, extended with two extra rules (including a Cut rule), to $\mathcal{E}$, and vice versa.

Let $\mathcal{L}$ be a non-empty signature for first-order classical logic, consisting of relation symbols and function symbols of finite (possibly zero) arity, and expand its attendant alphabet by the addition of a new logical symbol $\epsilon$. Terms and formulas are defined by mutual recursion, by adding to the standard formation clauses the following one: if $\varphi$ is a formula where $x$ has a free occurrence, then $\epsilon_{x}\varphi$ is a term whose free variables are the free variables in $\varphi$ minus $x$. Informally, $\epsilon_{x}\varphi$ denotes an $x$ that has the property $\varphi$ if any such object exists, and an arbitrary object otherwise. With a slight notational abuse, we refer to $\mathcal{L}$-terms and $\mathcal{L}$-formulas obtained by means of this expanded logical vocabulary as $\mathcal{L}^\mathcal{E}$-terms and $\mathcal{L}^\mathcal{E}$-formulas, respectively.

The sequent calculus $\mathcal{E}$ is a two-sided version of the one-sided calculus due to Leisenring \cite{Leisenring}. Its sequents have the form $\Gamma \seq \Delta$, where $\Gamma,\Delta$ are sets of $\mathcal{L}^\mathcal{E}$-formulas. Its structural and propositional logical introduction rules are the same as in $\mathcal{ST}^P$, with the addition of the following Cut rule:
\begin{prooftree}
\def\fCenter{\seq}
\Axiom$\Gamma \fCenter \Delta, \varphi$
\Axiom$\varphi, \Pi \fCenter \Sigma$\RightLabel{\scriptsize (CUT)}
\BinaryInf$\Gamma,\Pi \fCenter \Delta,\Sigma$
\end{prooftree}
Its rules involving quantifiers are listed in Figure~\ref{fig: rules-eps}. Crucially, Leisenring's calculus is not cut-free complete with respect to the Epsilon calculus, in its traditional presentation as a Hilbert-style calculus. Therefore, no implication as to the direct relationship between $\STH$ and the Epsilon calculus immediately follows from the results below.

\begin{figure}
\caption{Quantifier rules in $\mathcal{E}$}
\label{fig: rules-eps}

\medskip

\begin{prooftree}
\def\fCenter{\seq}
\Axiom$\varphi[x \mapsto \epsilon_{x}\varphi],  \Gamma \fCenter \Delta$\RightLabel{\scriptsize ($\exists$L$\epsilon$)}
\UnaryInf$\exists x\,\varphi,\Gamma \fCenter \Delta$

\def\fCenter{\seq}
\Axiom$\Gamma \fCenter \Delta, \varphi[x \mapsto  t]$\RightLabel{\scriptsize ($\exists$R$\epsilon$)}
\UnaryInf$\Gamma \fCenter \Delta, \exists x\,\varphi$

\noLine
\BinaryInfC{}
\end{prooftree}

\medskip

\begin{prooftree}

\def\fCenter{\seq}
\Axiom$\varphi[x \mapsto t], \Gamma \fCenter \Delta$\RightLabel{\scriptsize ($\forall$L$\epsilon$)}
\UnaryInf$\forall x\,\varphi, \Gamma \fCenter \Delta$

\def\fCenter{\seq}
\Axiom$\Gamma \fCenter \Delta, \varphi[x \mapsto \epsilon_{x}\lnot \varphi]$\RightLabel{\scriptsize ($\forall$R$\epsilon$)}
\UnaryInf$\Gamma \fCenter \Delta, \forall x \varphi$

\noLine
\BinaryInfC{}

\end{prooftree}

\medskip

\end{figure}

We now define by simultaneous induction a translation of $\mathcal{L}$-Henkin terms to $\mathcal{L}^{\mathcal{E}}$-terms, and of $\mathcal{L}$-Henkin formulas to $\mathcal{L}^{\mathcal{E}}$-formulas.

\begin{definition}\label{scimproriu}
    If $t$ is an $\mathcal{L}$-Henkin term, then $t^E$ is the $\mathcal{L}^{\mathcal{E}}$-term defined as follows:
    \begin{enumerate}
        \item $t^E := t$ if $t$ is a variable or a constant in $\mathcal{L}$;
        \item $t^E := \epsilon_{x} \varphi^E$ if $t$ is $\wit (\exists x\, \varphi)$;
        \item $t^E := \epsilon_{x} \lnot \varphi^E$ if $t$ is $\wit (\forall x\, \varphi)$;
        \item $t^E := f^{n}(t_1^{E},...,t_n^{E})$ if $t$ is $f^{n}(t_1,...,t_n)$.
    \end{enumerate}
    If $\varphi$ is a $\mathcal{L}$-Henkin formula, then $\varphi^E$ is the $\mathcal{L}^{\mathcal{E}}$-formula defined as follows:
    \begin{enumerate}
        \item $\varphi^E := P^n(t_1^{E},...,t_n^{E})$ if $\varphi$ is $P^n(t_1,...,t_n)$;
        \item $\varphi^E := \lnot (\psi^E)$ if $\varphi$ is $\lnot \psi$;
        \item $\varphi^E := \psi^E \circ \chi^E$ if $\varphi$ is $\psi \circ \chi$ ($\circ \in \{ \land, \lor \}$);
        \item $\varphi^E := Qx(\psi^E)$ if $\varphi$ is $Qx\psi$ ($Q \in \{ \exists, \forall \}$).
    \end{enumerate}
\end{definition}

Henceforth, for $\Gamma$ a set of $\mathcal{L}$-Henkin formulas, we use the abbreviation $\Gamma^E$ to denote $\{ \varphi^E \mid \varphi \in \Gamma \}$ and for $X$ a set of $\mathcal{L}$-Henkin sequents, we use $X^E$ to denote $ \{ \Gamma \seq \Delta \mid \Gamma^E \seq \Delta^E \in X \}$.

\begin{lemma}\label{careganzu}
 For any $\mathcal{L}$-Henkin formula $\varphi(x_1,...,x_n)$, we have that 
\[
(\varphi[x_1 \mapsto t_1,...,x_n \mapsto t_n])^E = \varphi^E[x_1 \mapsto t_1^E,...,x_n \mapsto t_n^E].
\]
 
\end{lemma}

\begin{proof}
    By induction on the construction of $\varphi$.
\end{proof}

Let $\mathcal{ST}^{HC}$ be the calculus obtained from $\STH$ by:
\begin{itemize}
    \item removing all the elimination rules (including (UWE) and (EWE));
    \item adding (CUT);
    \item adding the following bidirectional rules:
        \begin{prooftree}
        \def\fCenter{\seq}
        \Axiom$\varphi[x \mapsto \wit (\forall x\, \psi)], \Gamma \fCenter \Delta $\RightLabel{\scriptsize (WEXCHL)}
        \doubleLine
        \UnaryInf$\varphi[x \mapsto \wit (\exists x\, \lnot \psi)], \Gamma \fCenter \Delta$
    \end{prooftree}
    \begin{prooftree}
        \def\fCenter{\seq}
        \Axiom$\Gamma \fCenter \Delta, \varphi[x \mapsto \wit (\forall x\, \psi)]$\RightLabel{\scriptsize (WEXCHR)}
        \doubleLine
        \UnaryInf$\Gamma \fCenter \Delta, \varphi[x \mapsto \wit (\exists x\, \lnot \psi)]$
        \end{prooftree}
\end{itemize}

\begin{lemma}\label{bessiminci}
 The elimination rules of $\STH$ are derivable in $\mathcal{ST}^{HC}$.
\end{lemma}

\begin{proof}
    We confine ourselves to the witness elimination rules. The following proof trees derive the conclusions of (UWE) and (EWE) from their respective premises:
\begin{prooftree}
\def\fCenter{\seq}
\AxiomC{$\Gamma \fCenter \Delta, \varphi(\wit (\forall x\,\varphi))$}\RightLabel{\scriptsize (UWI)}
\AxiomC{$\varphi (t) \fCenter \varphi (t)$}
\UnaryInf$\varphi(\wit (\forall x\,\varphi)) \fCenter \varphi (t)$\RightLabel{\scriptsize (CUT)}
\BinaryInf$\Gamma \fCenter \Delta, \varphi (t)$
\end{prooftree}
\begin{prooftree}
\def\fCenter{\seq}
\AxiomC{$\varphi(\wit (\exists x\,\varphi)), \Gamma \fCenter \Delta$}\RightLabel{\scriptsize (EWI)}
\AxiomC{$\varphi (t) \fCenter \varphi (t)$}
\UnaryInf$\varphi (t) \fCenter \varphi(\wit (\exists x\,\varphi))$\RightLabel{\scriptsize (CUT)}
\BinaryInf$\varphi (t), \Gamma \fCenter \Delta $
\end{prooftree}
\end{proof}

In view of the foregoing lemma, we can view $ \mathcal{ST}^{HC}$ as an extension of $ \mathcal{ST}^{H}$.

\begin{theorem}\label{faccetontu}
    If $ X \vdash_{\mathcal{ST}^{HC}} \Gamma \seq \Delta$, then $ X^E \vdash_{\mathcal{E}} \Gamma^E \seq \Delta^E$.
\end{theorem}

\begin{proof}
    We proceed by induction on the size of a smallest derivation $\mathcal{D}$ of $\Gamma \seq \Delta$ from $X$ in $ \mathcal{ST}^{HC}$.
    
    Base. Suppose that $ \vdash_{\mathcal{ST}^{HC}} \varphi \seq \varphi$. By Definition \ref{scimproriu}, $\varphi^E$ is a $\mathcal{L}^E$-formula, whence $ \vdash_{\mathcal{E}} \varphi^E \seq \varphi^E$. Similarly, if $\Pi \seq \Sigma \in X$, $\Pi^E \seq \Sigma^E \in X^E$ and thus $ X^E \vdash_{\mathcal{E}} \Pi^E \seq \Sigma^E$.

    Inductive step. Observe that the cases of (WEXCHL) and (WEXCHR) are trivial. We check two of all the non-trivial remaining cases, the other ones being similar to either the former or the latter.

    As regards (EWI), by inductive hypothesis $ X^E \vdash_{\mathcal{E}} \Gamma^E \seq \Delta^E, (\varphi[x \mapsto t])^E$. By Lemma \ref{careganzu}, $(\varphi[x \mapsto t])^E = \varphi^E[x \mapsto t^E]$. Consider the following derivation from $X^E$ in $\mathcal{E}$:
\begin{prooftree}
\def\fCenter{\seq}
\AxiomC{\vdots}
\noLine\UnaryInf$\Gamma^E \fCenter \Delta^E, \varphi^E[x \mapsto t^E]$\RightLabel{\scriptsize ($\exists$R$\epsilon$)}
\UnaryInf$\Gamma^E \fCenter \Delta^E, \exists x\, \varphi^E$
\AxiomC{$\varphi^E [x \mapsto \epsilon_x \varphi^E] \fCenter \varphi^E [x \mapsto \epsilon_x \varphi^E] $}\RightLabel{\scriptsize ($\exists$L$\epsilon$)}
\UnaryInf$\exists x\, \varphi^E \fCenter \varphi^E [x \mapsto \epsilon_x \varphi^E]$\RightLabel{\scriptsize (CUT)}
\BinaryInf$ \Gamma^E \fCenter \Delta^E, \varphi^E [x \mapsto \epsilon_x \varphi^E]$
\end{prooftree}
By definition $\varphi^E [x \mapsto \epsilon_x \varphi^E]$ is $\varphi^E[x \mapsto \wit (\exists x\,\varphi)^E] = (\varphi[x \mapsto \wit (\exists x\,\varphi)])^E $, whence our conclusion follows.

We now proceed to the case of ($\exists$LW$\downarrow$). Applying again the IH, Lemma \ref{careganzu}, and Definition \ref{scimproriu}, we obtain that $X^E \vdash_{\mathcal{E}} \varphi^E (\epsilon_x \varphi^E), \Gamma^E \seq \Delta^E$, and by ($\exists$L$\epsilon$), $ X^E \vdash_{\mathcal{E}} \exists x\, \varphi^E, \Gamma^E \seq \Delta^E$.
\end{proof}

In the opposite direction, we define by simultaneous induction a translation of $\mathcal{L}^{\mathcal{E}}$-terms to $\mathcal{L}$-Henkin terms, and of $\mathcal{L}^{\mathcal{E}}$-formulas to $\mathcal{L}$-Henkin-formulas.

\begin{definition}\label{zugata}
    If $t$ is a $\mathcal{L}^{\mathcal{E}}$-term and $\varphi$ is a $\mathcal{L}^{\mathcal{E}}$-formula, then $t^{W}$ is the $\mathcal{L}$-Henkin term defined as follows:
    \begin{enumerate}
        \item $t^{W} := t$ if $t$ is a variable or a constant in $\mathcal{L}$;
        \item $t^{W} := \wit (\exists x\, \psi^{W})$ if $t$ is $\epsilon_{x} \psi$;
        \item $t^{W} := f^{n}(t_1^{W},...,t_n^{W})$ if $t$ is $f^{n}(t_1,...,t_n)$.
    \end{enumerate}
    If $\varphi$ is a $\mathcal{L}^{\mathcal{E}}$-formula, then $\varphi^E$ is the  $\mathcal{L}$-Henkin formula defined as follows:
    \begin{enumerate}
        \item $\varphi^W := P^n(t_1^{W},...,t_n^{W})$ if $\varphi$ is $P^n(t_1,...,t_n)$;
        \item $\varphi^W := \lnot (\psi^W)$ if $\varphi$ is $\lnot \psi$;
        \item $\varphi^W := \psi^W \circ \chi^W$ if $\varphi$ is $\psi \circ \chi$ ($\circ \in \{ \land, \lor \}$);
        \item $\varphi^W := Qx(\psi^W)$ if $\varphi$ is $Qx\psi$ ($Q \in \{ \exists, \forall \}$).
    \end{enumerate}
\end{definition}

Again, we will use the abbreviation $\Gamma^W$ to denote $\{ \varphi^W \mid \varphi \in \Gamma \}$ and $X^W$ to denote $ \{ \Gamma \seq \Delta \mid \Gamma^W \seq \Delta^W \in X \}$.

\begin{lemma}\label{murdegu}
 For any $\mathcal{L}^{\mathcal{E}}$-formula $\varphi(x_1,...,x_n)$, we have that 
\[
(\varphi[x_1 \mapsto t_1,...,x_n \mapsto t_n])^W = \varphi^W[x_1 \mapsto t_1^W,...,x_n \mapsto t_n^W].
\]
\end{lemma}

\begin{theorem}\label{frastimu}
    If $X \vdash_{\mathcal{E}}  \Gamma \seq \Delta$, then $ X^W \vdash_{\mathcal{ST}^{HC}} \Gamma^W \seq \Delta^W$.
\end{theorem}

\begin{proof}
    Following the footsteps of Theorem \ref{faccetontu}, we proceed by induction on the size of a smallest derivation $\mathcal{D}$ of $\Gamma \seq \Delta$ from $X$ in $ \mathcal{E}$. Our presentation of the different cases will be even more streamlined, in so far as we present only two subcases of the inductive step.
    
    As regards the rule ($\exists$R$\epsilon$), by inductive hypothesis $X^W \vdash_{\mathcal{ST}^{HC}} \Gamma^W \seq \Delta^W, (\varphi[x \mapsto t])^W$. By Lemma \ref{murdegu}, $X^W \vdash_{\mathcal{ST}^{HC}} \Gamma^W \seq \Delta^W, \varphi^W[x \mapsto t^{W}]$. Consider the following derivation from $X^W$:
\begin{prooftree}
\def\fCenter{\seq}
\AxiomC{\vdots}
\noLine\UnaryInfC{$\Gamma^W \seq \Delta^W, \varphi^W[x \mapsto t^{W}]$}\RightLabel{\scriptsize (EWI)}
\UnaryInfC{$\Gamma^W \seq \Delta^W, \varphi^W[x \mapsto \wit (\exists x\, \varphi^{W})]$}\RightLabel{\scriptsize ($\exists$RW)}
\UnaryInfC{$\Gamma^W \seq \Delta^W, \exists x\, \varphi^W$}
\end{prooftree}
By Definition \ref{frastimu}, $\exists x\, \varphi^W$ is $(\exists x\, \varphi)^W$, whence the conclusion of ($\exists$R$\epsilon$) has been established.

As regards the rule ($\exists$L$\epsilon$), $X^W \vdash_{\mathcal{ST}^{HC}} (\varphi[x \mapsto \epsilon_{x} \varphi(x)])^W, \Gamma^W \seq \Delta^W$ holds by the inductive hypothesis, i.e.\ $X^W \vdash_{\mathcal{ST}^{HC}} \varphi^W[x \mapsto \wit (\exists x\, \varphi^{W})], \Gamma^W \seq \Delta^W$. Hence by ($\exists$LW) $X^W \vdash_{\mathcal{ST}^{HC}} \exists x\, \varphi^{W}, \Gamma^W \seq \Delta^W$, which is again enough for our conclusion.
\end{proof}

\begin{theorem}
Let $\varphi$ be a $\mathcal{L}$-Henkin formula and $\psi$ be a $\mathcal{L}^{\mathcal{E}}$-formula. Then $\vdash_{\mathcal{ST}^{HC}} \varphi \seq \varphi^{EW}$ and $\vdash_{\mathcal{ST}^{HC}} \varphi^{EW} \seq \varphi$, and likewise $\vdash_{\mathcal{E}} \psi \seq \psi^{WE}$ and $\vdash_{\mathcal{E}} \psi^{WE} \seq \psi$.
\end{theorem}

\begin{proof}
    We confine ourselves to showing that $\vdash_{\mathcal{ST}^{HC}} \varphi \seq \varphi^{EW}$ for any $\mathcal{L}$-Henkin formula $\varphi$. This claim is proved by induction on the number $n$ of universal Henkin constants occurring in $\varphi$.

    If $n = 0$, then $\varphi$ is the same as $\varphi^{EW}$, and the claim follows from (ID). If $n > 0$, pick some $\wit (\forall x\, \psi)$ occurring in $\varphi$. The inductive hypothesis and Lemmas \ref{careganzu} and \ref{murdegu} give us a proof of the sequent $\varphi[x \mapsto  \wit (\exists x\, \lnot \psi)] \seq \varphi^{EW}[x \mapsto  \wit (\exists x\, \lnot \psi)^{EW}]$. Applying the Cut rule to this sequent and to the result of the derivation
\begin{prooftree}
\def\fCenter{\seq}
\AxiomC{$\varphi[x \mapsto  \wit (\forall x\, \psi)] \fCenter \varphi[x \mapsto  \wit (\forall x\, \psi)]$\RightLabel{\scriptsize (WEXCHR)}}
\UnaryInf$\varphi[x \mapsto  \wit (\forall x\, \psi)] \fCenter \varphi[x \mapsto  \wit (\exists x\, \lnot \psi)]$
\end{prooftree}
  yields a proof of the sequent $\varphi[x \mapsto  \wit (\forall x\, \psi)] \seq \varphi^{EW}[x \mapsto  \wit (\exists x\, \lnot \psi)^{EW}]$. However, since $\wit (\exists x\, \lnot \psi)^{EW}$ is the same as $\wit (\forall x\, \psi)^{EW}$, our conclusion follows by Lemmas \ref{careganzu} and \ref{murdegu}.
\end{proof}

\section{The Calculus $\MQST$}\label{marcopignotti}

\begin{figure}
\caption{Quantifier rules in $\MQST$}
\label{fig: rules-qst}

\medskip

\begin{prooftree}
\def\fCenter{\seq}
\Axiom$\exists x\,\varphi, \Gamma \fCenter \Delta$\RightLabel{\scriptsize ($\exists$LE)}
\UnaryInf$\varphi[x \mapsto t], \Gamma \fCenter \Delta$

\def\fCenter{\seq}
\Axiom$\Gamma \fCenter \Delta, \varphi[x \mapsto  t]$\RightLabel{\scriptsize ($\exists$RI)}
\UnaryInf$\Gamma \fCenter \Delta, \exists x\,\varphi$

\noLine
\BinaryInfC{}
\end{prooftree}

\medskip

\begin{prooftree}

\def\fCenter{\seq}
\Axiom$\varphi[x \mapsto t], \Gamma \fCenter \Delta$\RightLabel{\scriptsize ($\forall$LI)}
\UnaryInf$\forall x\,\varphi, \Gamma \fCenter \Delta$

\def\fCenter{\seq}
\Axiom$\Gamma \fCenter \Delta, \forall x\,\varphi$\RightLabel{\scriptsize ($\forall$RE)}
\UnaryInf$\Gamma \fCenter \Delta, \varphi[x \mapsto t]$

\noLine
\BinaryInfC{}

\end{prooftree}

\medskip

\begin{prooftree}
\def\fCenter{\seq}
\AxiomC{$\mathcal{D}$}
\noLine
\UnaryInfC{$\varphi[x \mapsto y], \Gamma \fCenter \Delta$}\RightLabel{\scriptsize ($\exists$LI)}
\UnaryInfC{$\exists x\,\varphi, \Gamma \fCenter \Delta$}

\AxiomC{$\mathcal{D}$}
\noLine
\UnaryInfC{$\Gamma \fCenter \Delta, \varphi[x \mapsto y]$}\RightLabel{\scriptsize ($\forall$RI)}
\UnaryInfC{$\Gamma \fCenter \Delta, \forall x\,\varphi$}

\noLine
\BinaryInfC{}
\end{prooftree}

\medskip

\begin{prooftree}
\def\fCenter{\seq}
\AxiomC{$\mathcal{D}_{1}$}
\noLine
\UnaryInfC{$\Gamma \fCenter \Delta, \exists x\,\varphi$}
\AxiomC{$[\Gamma \fCenter \Delta, \varphi[x \mapsto y]]$}
\noLine
\UnaryInfC{$\mathcal{D}_{2}$}
\noLine
\UnaryInfC{$\Pi \fCenter \Sigma$}\RightLabel{\scriptsize ($\exists$RE)}
\BinaryInfC{$\Pi \fCenter \Sigma$}

\AxiomC{$\mathcal{D}_{1}$}
\noLine
\UnaryInfC{$\forall x\,\varphi, \Gamma \fCenter \Delta$}
\AxiomC{$[\varphi[x \mapsto y], \Gamma \fCenter \Delta]$}
\noLine
\UnaryInfC{$\mathcal{D}_{2}$}
\noLine
\UnaryInfC{$\Pi \fCenter \Sigma$}\RightLabel{\scriptsize ($\forall$LE)}
\BinaryInfC{$\Pi \fCenter \Sigma$}

\noLine
\BinaryInfC{}
\end{prooftree}

\end{figure}

  The proof theory of the calculus $\STH$ does not appear to be very well-behaved, at least in so far as it is difficult to pin down a definition of a normal proof in the calculus, and consequently it is not clear what a normalisation theorem would have to look like. We now introduce a calculus which avoids the use of Henkin constants. This calculus includes rules that allow for the discharge of premise-sequents. This is not unprecedented in the literature: see \cite{Golan2, Hlobil}.

It is well known that calculi whose sequents are ordered pairs of \emph{sets} of formulas are not very amenable to an effective proof-theoretical analysis. With an eye to obtaining normalisation and interpolation theorems for our calculus, therefore, it is convenient to use a multiset calculus.

A (finite) \emph{multiset} of $\mathcal{L}$-formulas is, formally speaking, a function $\Gamma\colon Fm_\mathcal{L} \to \mathbb{N}$ (where $Fm_\mathcal{L}$ is the set of all $\mathcal{L}$-formulas) such that $\Gamma(\varphi) = 0$ for all but finitely many formulas $\varphi$. We think of $\Gamma(\varphi)$ as the multiplicity of the formula $\varphi$ in the multiset $\Gamma$. If $\Gamma$ is a multiset of $\mathcal{L}$-formulas, we write $\Gamma, \varphi$ for the multiset $\Delta$ such that $\Delta(\varphi) = \Gamma(\varphi) + 1$ and otherwise $\Delta(\psi) = \Gamma(\psi)$. We denote by $\lvert \Gamma \rvert$ the set $\{ \varphi \mid \Gamma (\varphi) > 0 \}$.

A \emph{sequent} in the calculus $\MQST$ is a pair of multisets of $\mathcal{L}$-formulas, written as $\Gamma \seq \Delta$. The structural rules of $\MQST$ are the axiom of Generalised Identity:
\vskip 10pt
\begin{prooftree}
\def\fCenter{\seq}
\AxiomC{}\RightLabel{\scriptsize (GID)}
\UnaryInf$\varphi, \Gamma \fCenter \Delta, \varphi$
\end{prooftree}
\vskip 10pt
and the rule of Contraction:
\begin{prooftree}
\def\fCenter{\seq}
\Axiom$\Gamma \fCenter \Delta, \varphi, \varphi$\RightLabel{\scriptsize (CR)}
\UnaryInf$\Gamma \fCenter \Delta, \varphi$

\Axiom$\varphi, \varphi, \Gamma \fCenter \Delta$\RightLabel{\scriptsize (CL)}
\UnaryInf$\varphi, \Gamma \fCenter \Delta$

\noLine
\BinaryInfC{}
\end{prooftree}
The introduction and elimination rules for the propositional connectives are listed in Figure~\ref{fig: rules}. That is, they are shared with the calculus $\STH$, except for the fact that we now interpret sequents as pairs of multisets. Finally, the introduction and elimination rules for the quantifiers are listed in Figure~\ref{fig: rules-qst}. These are subject to the usual eigenvariable restrictions familiar from the natural deduction calculi for classical and intuitionistic predicate logic. That is, in the rules ($\exists$LI) and ($\forall$RI) the variable $y$ must not occur freely in any of the undischarged premises of $\mathcal{D}$ or in $\Gamma \seq \Delta$. Similarly, in the rules ($\exists$RE) and ($\forall$LE) the variable $y$ must not occur freely in any of the undischarged premises of $\mathcal{D}_{2}$ or in $\Gamma \seq \Delta$ or $\Pi \seq \Sigma$. These last two rules discharge any number of instances of the premise $\Gamma \seq \Delta, \varphi[x \mapsto y]$ and $\varphi[x \mapsto y], \Gamma \seq \Delta$, respectively (as indicated by the outer brackets in Figure~\ref{fig: rules-qst}.)

An \emph{atomic} instance of (GID) is a sequent of the form $\varphi, \Gamma \seq \Delta, \varphi$ where $\varphi$ is an atomic formula. A \emph{weakening} of a sequent $\Gamma \seq \Delta$ is a sequent of the form $\Gamma, \Gamma' \seq \Delta, \Delta'$. A \emph{derivation} of $S$ from a set of sequents $X$ in $\MQST$ is an appropriately labelled proof tree where the conclusion is $S$ and each undischarged assumption is a sequent in $X$. The following two lemmas are straightforward to prove by induction over the complexity of proofs.

\begin{lemma} \label{lemma: atomic identity}
  Every instance of (GID) is derivable in~$\MQST$ using only atomic instances of (GID) and introduction rules.
\end{lemma}

\begin{lemma} \label{lemma: renaming variables}
  Let $\mathcal{D}$ be a derivation of $S$ from $X$, and let $z$ be a variable which does not occur anywhere in $\mathcal{D}$. Then substituting all free occurrences of a variable $x$ by $z$ throughout $\mathcal{D}$ yields a derivation of $S[x \mapsto z]$ from $X[x \mapsto z]$.
\end{lemma}

\begin{lemma} \label{lemma: weakening admissible}
  The rule of Weakening is admissible in~$\MQST$:
\begin{prooftree}
\def\fCenter{\seq}
\Axiom$\Gamma \fCenter \Delta$\RightLabel{\scriptsize (WL)}
\UnaryInf$\varphi, \Gamma \fCenter \Delta$

\def\fCenter{\seq}
\Axiom$\Gamma \fCenter \Delta$\RightLabel{\scriptsize (WR)}
\UnaryInf$\Gamma \fCenter \Delta, \varphi$

\noLine
\BinaryInfC{}
\end{prooftree}
  In particular, if $S$ is derivable from $X$, then so is every weakening of $S$.
\end{lemma}

\begin{proof}
  This is a straightforward induction on the depth of the derivation of $S$ from $X$, using Lemma~\ref{lemma: renaming variables} to rename variables in the case of rules with eigenvariable conditions. For example, if $S$ is the sequent $\Gamma \seq \Delta, \forall x\,\varphi$ and the last step of the proof infers $S$ from $\Gamma \seq \Delta, \varphi[x \mapsto z]$, fix multisets of formulas $\Gamma',\Delta'$ and consider a variable $z'$ which occurs neither in $S$ nor in $X$, $ \lvert \Gamma \rvert$, $\lvert \Gamma' \rvert$, $\lvert \Delta \rvert$, $\lvert \Delta' \rvert$. Then by Lemma~\ref{lemma: renaming variables} there is a derivation of $\Gamma \seq \Delta, \varphi[x \mapsto z']$ from $X$, so by the inductive hypothesis there is a derivation of $\Gamma, \Gamma' \seq \Delta, \Delta', \varphi[x \mapsto z']$, and we can now infer $\Gamma, \Gamma' \seq \Delta, \Delta', \forall x\,\varphi$ from this sequent.
\end{proof}

  The proof of the soundness theorem is a straightforward induction on the complexity of a proof entirely analogous to the soundness theorem for classical or intuitionistic logic. In its statement we use the following notation: given a sequent $S := \Gamma \seq \Delta$ and a set of sequents $X$, we define $\lvert S \rvert$ as $ \lvert \Gamma \rvert \seq \lvert \Delta \rvert$ and $\lvert X \rvert$ as $\{ \lvert S \rvert \mid S \in X \}$.

  \begin{theorem}
      If $X \vdash_{\MQST} S$, then $\lvert X \rvert \models_{\mathrm{ST}} \lvert S \rvert$.
  \end{theorem}

  Hereafter, we use the notation
  \begin{align*}
   \Gamma[x \mapsto t] & \assign \set{\gamma[x \mapsto t]}{\gamma \in \Gamma}, \\
   (\Gamma \seq \Delta)[x \mapsto t] & \assign \Gamma[x \mapsto t] \seq \Delta[x \mapsto t], \\
   X[x \mapsto t] & \assign \set{S[x \mapsto t]}{S \in X}.
  \end{align*}
  In particular, the notation $X[x \mapsto t]$ assumes that in each formula in each sequent in $X$ the term $t$ can be substituted for $x$.
  
  \begin{lemma} \label{lemma: substitution}
    If $X \vdash_{\MQST} S$, then $X[x \mapsto t] \vdash_{\MQST} S[x \mapsto t]$.
  \end{lemma}

  \begin{proof}
    Consider a derivation $\mathcal{D}$ of $S$ from $X$ and pick a topmost application (if any such exists) of one of the rules involving the eigenvariable condition, say the inference from $\Gamma \seq \Delta, \varphi[x \mapsto y]$ to $\Gamma \seq \Delta, \forall x\,\varphi$. Since the variable $y$ does not occur among the premises of the proof, we may uniformly substitute a fresh variable $z$ not occurring in either $\mathcal{D}$ or $t$ for $y$ throughout the subderivation which ends with that application. Repeating this process results in a derivation of $S$ from $X$ where the eigenvariable restrictions only apply to variables which do not occur in $t$. We may now substitute $t$ for $x$ throughout this derivation.
  \end{proof}

    The following lemma states that the sequent counterpart of the standard rule of disjunction elimination is derivable in $\MQST$. We again use the notation
\begin{align*}
  & (\Gamma_{1} \seq \Delta_{1}) \sqcup (\Gamma_{2} \seq \Delta_{2}) \assign \Gamma_{1}, \Gamma_{2} \seq \Delta_{1}, \Delta_{2}, & & X \sqcup S \assign \set{S' \sqcup S}{S' \in X}.
\end{align*}

  \begin{lemma}\label{lemma: pcp for qst}
  Suppose that
  \[
  X, S_{1} \vdash_{\MQST} S \text{ and } X, S_{2} \vdash_{\MQST} S.
  \]
  Then
  \[
  X, S_{1} \sqcup S_{2} \vdash_{\MQST} S.
  \]
\end{lemma}

\begin{proof}
  Consider derivations $\mathcal{D}_{1}$ and $\mathcal{D}_{2}$ of $S$ from $X, S_{1}$ and $X, S_{2}$, respectively. Let $S_{1}'$ ($S_{2}'$, $S'$) be obtained from $S_{1}$ (from $S_{2}$, from $S$) by replacing distinct variables in $S_{1}$ by distinct fresh variables which do not occur anywhere in $\mathcal{D}_{1}$ and $\mathcal{D}_{2}$.

  Prefacing every premise of $\mathcal{D}_{1}$ with instances of (WL),(WR) yields a derivation $X \sqcup S_{2}', S_{1} \sqcup S_{2}' \vdash_{\MQST} S \sqcup S_{2}'$. (Because the free variables on $S_{2}'$ do not occur anywhere in $\mathcal{D}_{1}$, this does not interfere with any of the quantifier rules.) By Lemma~\ref{lemma: substitution} we have $X \sqcup S_{2}, S_{1} \sqcup S_{2} \vdash_{\MQST} S \sqcup S_{2}$. Similarly, prefacing $\mathcal{D}_{2}$ with instances of (WL),(WR) yields a derivation $X \sqcup S', S_{2} \sqcup S' \vdash_{\MQST} S \sqcup S'$. By Lemma~\ref{lemma: substitution} we have $X \sqcup S, S_{2} \sqcup S \vdash_{\MQST} S \sqcup S$. Concatenating these one obtains a derivation $X \sqcup S, X \sqcup S_{2}, S_{1} \sqcup S_{2} \vdash_{\MQST} S \sqcup S$. Since $X \vdash_{\MQST} S'$ for each $S' \in X \sqcup S$, and $X \vdash_{\MQST} S''$ for each $S'' \in X \sqcup S_{2}$, and $S \sqcup S$ is $S$, we obtain that $X, S_{1} \sqcup S_{2} \vdash_{\MQST} S$.
\end{proof}

In order to obtain completeness via a canonical model, we have to patiently reorganise the entire apparatus of theories we used in the first part of this paper, for now we no longer have Henkin constants at our disposal. For a start, we inductively define the notion of a \emph{witness} for an $\mathcal{L}$-sequent $S$. 

\begin{definition}
An \emph{$\mathcal{L}$-theory} (for short, a \emph{theory}, when no confusion is likely to arise) is a set of $\mathcal{L}$-sequents closed under derivability in $\MQST$.
\end{definition}

Since all the definitions contained in Definition \ref{cataratta} are unaffected by the change of language, we will freely use them in what follows. Moreover, we add the following
\begin{definition}
A theory $T$ is \emph{witnessed} if
\begin{enumerate}
\item $\Gamma \seq \Delta, \exists x\,\varphi \in T$ implies $\Gamma \seq \Delta, \varphi[x \mapsto t] \in T$ for some term $t$, and 
\item $\forall x\,\varphi, \Gamma \seq \Delta \in T$ implies $\varphi[x \mapsto t], \Gamma \seq \Delta \in T$ for some term $t$.
\end{enumerate}
\end{definition}

Hereafter, the set of all $\mathcal{L}$-variables will be denoted by $Var_{\mathcal{L}}$. The set of variables which occur (whether free or bound) in a set of sequents $X$ will be denoted by $Var(X)$. We say that $X$ \emph{contains few variables} if $\card{Var(X)} < \card{Var_{\mathcal{L}}}$. (Although we have previously used $\card{\Gamma}$ to denote the underlying set of a multiset $\Gamma$, here we use $\card{X}$ to denote the cardinality of the set $X$.)

\begin{theorem}\label{giubileo}
  Suppose that $\card{Var_{\mathcal{L}}} = \card{Seq_{\mathcal{L}}}$ is a regular cardinal. If $T$ is an $S$-consistent theory that contains few variables, then there exists a prime, witnessed and $S$-consistent theory $T^{\prime}$ such that $T\subseteq T^{\prime}$.
\end{theorem}

\begin{proof}
  Suppose that $T$ is $S$-consistent and contains few variables. We put the set of $\mathcal{L}$-sequents $Seq_{\mathcal{L}}$ in correspondence with some cardinal $\kappa$, i.e.\ we consider a sequence of $\mathcal{L}$-sequents $S_{\alpha}$ with $\alpha \in \kappa$ which contains every $\mathcal{L}$-sequent. We now define a sequence of sets of sequents $Y_{\alpha}$ with $\alpha \in \kappa$ such that each $Y_{\alpha}$ is $S$-consistent and each $Y_{\alpha}$ contains few variables.
  
  We take $Y_{0} \assign T$. If $Y_{\alpha}, S_{\alpha} \vdash_{\MQST} S$, we take $Y_{\alpha+1} \assign Y_{\alpha}$. Otherwise, $Y_{\alpha}, S_{\alpha} \nvdash_{\MQST} S$. If $S_{\alpha}$ does not have the form $\Gamma \seq \Delta, \exists x\,\varphi$ or $\forall x\,\varphi, \Gamma \seq \Delta$, we take $Y_{\alpha+1} \assign Y_{\alpha} \cup \{ S_{\alpha} \}$. Otherwise, for each of the finitely many ways of writing $S$ as either $\Gamma \seq \Delta, \exists x\,\varphi$ or $\forall x\,\varphi, \Gamma \seq \Delta$ we pick a distinct fresh variable $y$ which does not occur in $Y_{\alpha} \cup \{ S_{\alpha}, S \}$ and we take $W_{\alpha}$ to be the non-empty finite set consisting of the sequents $\Gamma \seq \Delta, \varphi[x \mapsto y]$ and $\varphi[x \mapsto y], \Gamma \seq \Delta$ obtained in this way. (Such fresh variables exist because $Y_{\alpha}$ contains few variables and $\card{Var_{\mathcal{L}}} = \card{Seq_{\mathcal{L}}}$.) By ($\exists$RE) and ($\forall$LE), $Y_{\alpha}, S_{\alpha} \nvdash_{\MQST} S$ implies $Y_{\alpha}, S_{\alpha}, W_{\alpha} \nvdash_{\MQST} S$. Taking $Y_{\alpha+1} \assign Y_{\alpha} \cup \{ S_{\alpha} \} \cup W_{\alpha}$ therefore yields an $S$-consistent set. Because we only added finitely many sequents to $Y_{\alpha}$, the set $Y_{\alpha+1}$ still contains few variables.

  If $\alpha$ is a limit ordinal, we take $Y_{\alpha} \assign \bigcup \set{Y_{\beta}}{\beta < \alpha}$. Because $\vdash_{\MQST}$ is a finitary relation, $Y_{\alpha}$ is an $S$-consistent set. $Y_{\alpha}$ still contains few variables, since $Var(Y_{\alpha}) = \bigcup \set{Var(Y_{\beta})}{\beta < \alpha}$ where $\alpha < \card{Var_{\mathcal{L}}}$ and $Var(Y_{\beta}) < \card{Var_{\mathcal{L}}}$ for each $\beta < \alpha$. (This relies on the regularity of $\card{Var_{\mathcal{L}}}$.)
  
  Finally, we take $T' \assign \bigcup \set{Y_{\alpha}}{\alpha \in \kappa}$. Again, because $\vdash_{\MQST}$ is a finitary relation, $T' \nvdash_{\MQST} S$. The set $T'$ is a theory: if $T' \vdash_{\MQST} S_{\alpha}$, then clearly $Y_{\alpha}, S_{\alpha} \nvdash_{\MQST} S$, so $S_{\alpha} \in Y_{\alpha+1} \subseteq T'$. It is a witnessed theory by construction. Finally, $T'$ is a prime theory by Lemma~\ref{lemma: pcp for qst}. 
\end{proof}

\begin{theorem}
  Let $X \cup \{ S \}$ be a set of $\mathcal{L}$-sequents. If $\lvert X \rvert  \models_{\mathrm{ST}} \lvert S \rvert$, then $X \vdash_{\MQST} S$.
\end{theorem}

\begin{proof}
  Observe that we can always expand our set of variables $Var_{\mathcal{L}}$ to some larger set of variables $Var'_{\mathcal{L}} \supseteq Var_{\mathcal{L}}$ which satisfies the assumption of the previous theorem. Expanding the set of variables in this way changes neither the relation $\models_{\mathrm{ST}}$ nor the relation $\vdash_{\MQST}$. We may therefore assume without loss of generality that indeed $\card{Var_{\mathcal{L}}} = \card{Seq_{\mathcal{L}}}$ and that this cardinal is regular.

  Because expanding $\mathcal{L}$ by new object variables yields a consequence relation $\models'_{\mathrm{ST}}$ and a provability relation $\vdash'_{\MQST}$ which are conservative extensions of $\models_{\mathrm{ST}}$ and $\vdash_{\MQST}$, it suffices to prove the required implication in some extension of $\mathcal{L}$ by new variables. We may therefore assume without loss of generality that $X$, and hence $Th(X)$, contain few variables.
  
  Suppose contrapositively that $X \nvdash_{\MQST} S$. Then $\Th(X)$ is $S$-consistent and contains few variables. By Theorem \ref{giubileo}, there exists a prime, witnessed and $S$-consistent theory $T$ such that $\Th(X)\subseteq T$. We want to construct a canonical model $\mathsf{M}= \langle D, I\rangle$ such that $\mathsf{M}\models_{\mathrm{ST}} S'$ for all $S' \in T$, but it is not the case that $\mathsf{M}\models_{\mathrm{ST}} S$. Let thus:
  
  \begin{itemize}
    \item $D$ be the set of all terms;
    \item $I(x) = x$ for every variable $x$;
    \item $I(f^n)$, for any $n$-ary function symbol $f^n$, be the free operation associated with $f^n$:
    \begin{align*}
    f^{\mathsf{M}}\colon \langle t_{1}, \dots, t_{n} \rangle \mapsto f^n (t_{1}, \dots, t_{n});
    \end{align*} 
    \item for any $n$-ary relation symbol $P^n$,
    \[
    I(P^n)(t_{1}, \dots, t_{n}) =\begin{cases}
    1  & \text{ if } \emptyset \seq P^n(t_1,...,t_n) \in T, P^n(t_1,...,t_n) \seq \emptyset \notin T; \\
    \frac{1}{2} & \text{ if } \emptyset \seq P^n(t_1,...,t_n) \in T, P^n(t_1,...,t_n) \seq \emptyset \in T; \\
    0  & \text{ if } \emptyset \seq P^n(t_1,...,t_n) \notin T, P^n(t_1,...,t_n) \seq \emptyset \in T.
    \end{cases}
    \]
\end{itemize}
  Observe that, according to the previous definition, $I(t) = t$ for any $\mathcal{L}$-term $t$.
  
 Using the same strategy as in Theorem \ref{thm: completeness} we show that, for any $\mathcal{L}$-sequent $S'$, $\mathsf{M}$ $\mathrm{ST}$-satisfies $S'$ if and only if $S' \in T$. Hence $\mathsf{M}$ $\mathrm{ST}$-satisfies $X$, but it does not $\mathrm{ST}$-satisfy $S$, which means that $\lvert X \rvert \nvDash_{\mathrm{ST}} \lvert S \rvert$. 
\end{proof}

\section{Normalisation and Interpolation for $\mathcal{MQST}$}\label{antonellomura}

  Our goal is now to prove a normalisation theorem for $\MQST$ and then deduce the interpolation theorem as a corollary. Our argument is a slightly adapted version of the normalisation proof for natural deduction by Troelstra \& Schwichtenberg~\cite{TroelstraSchwichtenberg}.

  There are four main differences between their argument and ours. First, the rules that need to be discussed separately in the definition of a segment in $\MQST$ are the right existential and left universal elimination rules, rather than the disjunction and existential elimination rules. In $\MQST$ the disjunction and conjunction elimination rules behave entirely symmetrically. Second, the definition of a track in fact simplifies in $\MQST$ because there is no analogue of the elimination rule for implication (where the minor premise and the conclusion may be unrelated formulas). This means that every track ends at the root of the derivation. Third, since $\MQST$ operates on sequents, it may happen that an introduction rule is immediately followed by an elimination rule but these rules operate on different principal formulas. This does not introduce any substantial difficulty, but it does mean that an additional case needs to be discussed. Last, in addition to segments which start with an introduction and end with an elimination, we also count segments which start with an instance of (GID) and end with an elimination as cut segments.

  Let us use the term \emph{sidetrack rules} to refer to the rules ($\exists$RE) and ($\forall$LE). These are the only elimination rules where the conclusion is not obtained from the premises by peeling off the principal logical operator. In a sidetrack rule, the premise containing the formula whose principal operator is being eliminated will be called the \emph{major premise} and the other premise (which consists of the same sequent as the conclusion of the rule) will be called the \emph{minor premise}. In all other elimination rules both premises count as major premises.

\begin{definition}
    A \emph{segment} in a derivation is a non-empty sequence $S_{1}, \dots, S_{n}$ of sequent occurrences such that
\begin{enumerate}
\item for $1 \leq i \leq n-1$ either $S_{i}$ is the minor premise and $S_{i+1}$ is the conclusion of a sidetrack rule, or $S_{i}$ is the premise and $S_{i+1}$ is the conclusion of a contraction,
\item $S_{1}$ is not the conclusion of a sidetrack rule or a contraction,
\item $S_{n}$ is not the minor premise of a sidetrack rule or the premise of a contraction.
\end{enumerate} 
The \emph{rank} of a segment is the number of logical symbols that occur in its first sequent.
\end{definition}

  Each sequent in a derivation belongs to some segment, possibly one of length $1$. Up to contraction, a segment contains only instances of the same sequent.

\begin{definition}
Let:
    \begin{enumerate}
        \item a \emph{cut segment} be a segment where $S_{1}$ is either the conclusion of an introduction rule or an instance of (GID) and where $S_{n}$ is the premise (necessarily a major premise) of an elimination rule;
        \item a \emph{cut-free  derivation} be a derivation with no cut segment;
        \item  a \emph{normal derivation} be a cut-free derivation where every instance of a sidetrack rule discharges at least one assumption, and every instance of (GID) is atomic.
    \end{enumerate}
\end{definition}

\begin{theorem}
  If there is a derivation of a sequent $S$ from a set of sequents $X$ in~$\MQST$, then there is a normal derivation of $S$ from $X$ in~$\MQST$.
\end{theorem}

\begin{proof}
  We have already seen (Lemma~\ref{lemma: atomic identity}) that (GID) can be restricted to atomic instances. Moreover, each (normal) derivation is easily transformed into (a normal) one where moreover every instance of a sidetrack rule discharges at least one assumption: it suffices to omit all instances of sidetrack rules which discharge no assumption. To prove the theorem, it thus suffices to show that each such derivation can be transformed into a \emph{cut-free} derivation (a derivation with no cut segments).

  We prove this claim by double induction over the maximal rank $r$ of cut segments in the derivation and the sum $m$ of the lengths of all cut segments of maximal rank.

  Given a derivation with a cut segment, pick a topmost cut segment $\sigma$ of maximal rank. If this segment has length $n \geq 2$, we can obtain a derivation with the same $r$ but smaller $m$ by permuting the elimination rule immediately following $\sigma$ above the last sidetrack rule of $\sigma$. If necessary, we rename the variable involved in the sidetrack rule using Lemma~\ref{lemma: renaming variables}. (This corresponds to the ``permutation contractions'' of~\cite[6.1.3]{TroelstraSchwichtenberg}.)  For example, the following derivation (where the term $t$ might contain the variable $y$):
\begin{prooftree}
\def\fCenter{\seq}
\AxiomC{$\mathcal{D}_{1}$}
\noLine
\UnaryInfC{$\Gamma \fCenter \Delta, \exists x\,\varphi$}
\AxiomC{$[\Gamma \fCenter \Delta, \varphi[x \mapsto y]]$}
\noLine
\UnaryInfC{$\mathcal{D}_{2}$}
\noLine
\UnaryInfC{$\Pi \fCenter \Sigma, \forall u\,\varphi$}
\BinaryInfC{$\Pi \fCenter \Sigma, \forall u\,\varphi$}
\UnaryInfC{$\Pi \fCenter \Sigma, \varphi[u \mapsto t]$}
\end{prooftree}
  is transformed into the following derivation (where $z$ does not occur in $\mathcal{D}_{2}$)
\begin{prooftree}
\def\fCenter{\seq}
\AxiomC{$\mathcal{D}_{1}$}
\noLine
\UnaryInfC{$\Gamma \fCenter \Delta, \exists x\,\varphi$}
\AxiomC{$[\Gamma \fCenter \Delta, \varphi[x \mapsto z]]$}
\noLine
\UnaryInfC{$\mathcal{D}'_{2}$}
\noLine
\UnaryInfC{$\Pi \fCenter \Sigma, \forall u\,\varphi$}
\UnaryInfC{$\Pi \fCenter \Sigma, \varphi[u \mapsto t]$}
\BinaryInfC{$\Pi \fCenter \Sigma, \varphi[u \mapsto t]$}
\end{prooftree}
Because we now have (CL) and (CR), there is one additional type of permutation contraction. For example
\begin{prooftree}
\def\fCenter{\seq}
\AxiomC{$\mathcal{D}$}
\noLine
\UnaryInf$\Gamma \fCenter \Delta, \varphi \wedge \psi, \varphi \wedge \psi$
\UnaryInf$\Gamma \fCenter \Delta, \varphi \wedge \psi$
\UnaryInf$\Gamma \fCenter \Delta, \varphi$
\end{prooftree}
  is transformed into
\begin{prooftree}
\def\fCenter{\seq}
\AxiomC{$\mathcal{D}$}
\noLine
\UnaryInf$\Gamma \fCenter \Delta, \varphi \wedge \psi, \varphi \wedge \psi$
\UnaryInf$\Gamma \fCenter \Delta, \varphi, \varphi \wedge \psi$
\UnaryInf$\Gamma \fCenter \Delta, \varphi, \varphi$
\UnaryInf$\Gamma \fCenter \Delta, \varphi$
\end{prooftree}
  which again decreases $m$.

  If the segment $\sigma$ has length $1$ and is preceded by an instance $\varphi, \Gamma \seq \Delta, \varphi$ of (GID), then the elimination rule following $\sigma$ can be eliminated entirely: because all instances of (GID) are assumed to be atomic, this elimination rule must operate on $\Gamma$ or $\Delta$, so we can simply use a different instance of (GID).

 If the segment $\sigma$ has length $1$ and the principal formula of the introduction rule preceding $\sigma$ is different (or on a different side) than the principal formula of the elimination rule following $\sigma$, we permute the elimination rule above the introduction rule. For example, the derivation
\begin{prooftree}
\def\fCenter{\seq}
\AxiomC{$\mathcal{D}_{1}$}
\noLine
\UnaryInf$\neg \chi, \Gamma \fCenter \Delta, \varphi$
\AxiomC{$\mathcal{D}_{2}$}
\noLine
\UnaryInf$\neg \chi, \Gamma \fCenter \Delta, \psi$
\BinaryInf$\neg \chi, \Gamma \fCenter \Delta, \varphi \wedge \psi$
\UnaryInf$\Gamma \fCenter \Delta, \varphi \wedge \psi, \chi$
\end{prooftree}
  is replaced by
\begin{prooftree}
\def\fCenter{\seq}
\AxiomC{$\mathcal{D}_{1}$}
\noLine
\UnaryInf$\neg \chi, \Gamma \fCenter \Delta, \varphi$
\UnaryInf$\Gamma \fCenter \Delta, \varphi, \chi$
\AxiomC{$\mathcal{D}_{2}$}
\noLine
\UnaryInf$\neg \chi, \Gamma \fCenter \Delta, \psi$
\UnaryInf$\Gamma \fCenter \Delta, \psi, \chi$
\BinaryInf$\Gamma \fCenter \Delta, \varphi \wedge \psi, \chi$
\end{prooftree}
  Observe that the sequents $\neg \chi, \Gamma \seq \Delta, \varphi$ and $\neg \chi, \Gamma \seq \Delta, \psi$ and $\Gamma \seq \Delta, \varphi \wedge \psi, \chi$ may be part of new cut segments, but these cut segments all have strictly lower rank. If $\neg \chi, \Gamma \seq \Delta, \varphi \wedge \psi$ was the only cut segment of maximal rank, this transformation therefore decreases $r$. Otherwise, it decreases $m$.

  Finally, it remains to deal with cut segments of length $1$ where the principal formula of the introduction rule preceding the sequent is the same (and on the same side) as the principal formula of the elimination rule following the sequent. For example, with ($\land$R), the derivation
\begin{prooftree}
\def\fCenter{\seq}
\AxiomC{$\mathcal{D}_{1}$}
\noLine
\UnaryInf$\Gamma \fCenter \Delta, \varphi$
\AxiomC{$\mathcal{D}_{2}$}
\noLine
\UnaryInf$\Gamma \fCenter \Delta, \psi$
\BinaryInf$\Gamma \fCenter \Delta, \varphi \wedge \psi$
\UnaryInf$\Gamma \fCenter \Delta, \varphi$
\end{prooftree}
  is transformed into the derivation
\begin{prooftree}
\def\fCenter{\seq}
\AxiomC{$\mathcal{D}_{1}$}
\noLine
\UnaryInf$\Gamma \fCenter \Delta, \varphi$
\end{prooftree}
  Again, if $\Gamma \seq \Delta, \varphi \wedge \psi$ was the only cut segment of maximal rank, this trans\-formation decreases $r$. Otherwise, it decreases $m$.
\end{proof}

\begin{definition}\label{traccas}
 A \emph{track} in a normal derivation $\mathcal{D}$ is a sequence of occurrences of sequents $S_{1}, \dots, S_{n}$ such that $S_{1}$ is either an instance of (GID) or a top assumption not discharged by any application of a sidetrack rule, $S_{n}$ is the conclusion of $\mathcal{D}$, and for $1 \leq i \leq n-1$ either (i) $S_{i}$ is not the major premise of a sidetrack rule and $S_{i+1}$ is the sequent occurrence below $S_{i}$ or (ii) $S_{i}$ is the major premise of a sidetrack rule and $S_{i+1}$ is an assumption discharged by this rule. 
\end{definition}
  Observe that the existence of an assumption discharged by each application of a sidetrack rule is part of the definition of a normal derivation.
\begin{lemma}
  Each sequent in a normal derivation belongs to some track.
\end{lemma}

\begin{proof}
  Given a sequent $S$, the next sequent in any track involving $S$ is uniquely determined by conditions (i) and (ii) in Definition \ref{traccas}. The previous sequent may not be uniquely determined, but either $S$ is a premise discharged by an application of a sidetrack rule, in which case the previous sequent is the major premise of that rule, or it is the conclusion of a rule, in which case at least one of the premises of $S$ can play the role of the previous sequent in a track, or it is (GID) or an undischarged assumption, in which case the track starts with $S$.
\end{proof}

\begin{lemma} \label{lemma: midsegment}
  For each track of a normal derivation consisting of segments $\sigma_{1}, \dots, \sigma_{n}$ in this order there is a segment $\sigma_{i}$, called the \emph{midsegment} of the track, such that (i) the last sequent of each of the segments $\sigma_{1}, \dots, \sigma_{i-1}$ is the major premise of an elimination rule, (ii) the last sequent of each of the segments $\sigma_{i}, \dots, \sigma_{n-1}$ is the premise of an introduction rule, (iii) if the midsegment is not $\sigma_{1}$, then the first sequent of $\sigma_{1}$ is not the conclusion of (GID).
\end{lemma}

\begin{proof}
  If no such midsegment exists, then there is by definition a cut segment on this track.
\end{proof}

\begin{lemma} \label{lemma: reduction to single sequent}
  Each finite set of sequents $X$ is interderivable in $\MQST$ with a sequent of the form $\emptyset \seq \varphi$ such that $\varphi$ contains the same free variables and the same relation symbols as $X$.
\end{lemma}

\begin{proof}
    Each sequent $\varphi_1, \dots,\varphi_m \seq \psi_1, \dots, \psi_n$ is interderivable with the sequent $\emptyset \seq \neg \varphi_1 \vee \dots \vee \neg \varphi_m \vee \psi_1 \vee \dots \vee \psi_n$. Moreover, each finite set of sequents of the form $ \emptyset \seq \varphi_1, \dots, \emptyset \seq \varphi_n \}$ is interderivable with the sequent $\emptyset \seq \varphi_1 \wedge \dots \wedge \varphi_n$.
\end{proof}

\begin{theorem}
  If $X_{1}, X_{2} \vdash_{\MQST} S$, then there is a finite set of sequents $I$ such that $X_{1} \vdash_{\MQST} I$ and $I, X_{2} \vdash_{\MQST} S$ and $I$ only contains those relation symbols and those free variables which occur in both $X_{1}$ and in $X_{2} \cup \{ S \}$.
\end{theorem}

\begin{proof}
  In the course of the proof, we shall need to talk about universally (existentially) quantifying with respect to a variable $x$ in a sequent $T$. What we mean by this is that we first replace $T$ by an equivalent sequent of the form $\emptyset \seq \varphi$ using Lemma~\ref{lemma: reduction to single sequent} and then, if $x$ occurs free in $\varphi$, we universally (existentially) quantify $\varphi$ with respect to this variable.

  We prove the claim by induction on the size of a normal derivation $\mathcal{D}$ of $S$ from $X \assign X_{1} \cup X_{2}$. We shall call $I$ an \emph{interpolating set} for $X_{1} \mid X_{2} \vdash_{\MQST} S$. If the derivation does not contain any deductive steps, then $S \in X_1$ or $S \in X_2$. In the former case, $I \assign \{ S \}$ is an interpolating set, while in the latter case $I \assign \emptyset$ is an interpolating set. If the last step is an instance of (GID), then this is in fact the only step of the proof and we can take $I$ to be the empty set. If the last step of this derivation is a contraction from $S$ to $S'$, then each interpolating set $I$ for $X_{1} \mid X_{2} \vdash_{\MQST} S$ is also an interpolating set for $X_{1} \mid X_{2} \vdash_{\MQST} S'$. The same holds if the last step is any of ($\lnot$L$\downarrow$), ($\lnot$R$\downarrow$), ($\land$L$\downarrow$) or ($\lor$R$\downarrow$), since these are invertible rules with a single premise.

  It remains to deal with the cases where the last step is any of ($\forall$LI), ($\forall$RI), ($\exists$LI), ($\exists$RI), ($\land$R$\downarrow$) or ($\lor$L$\downarrow$). Using the inductive hypothesis, we only deal with three of these last six cases, since the other three are entirely analogous:
\begin{enumerate}
\item Let $I_{1}$ and $I_{2}$ be interpolating sets for $X_{1} \mid X_{2} \vdash_{\MQST} \Gamma \seq \Delta, \varphi$ and $X_{1} \mid X_{2} \vdash_{\MQST} \Gamma \seq \Delta, \psi$. Let $I \assign I_{1} \cup I_{2}$. 
Then $I$ is an interpolating set for $X_{1} \mid X_{2} \vdash_{\MQST} \Gamma \seq \Delta, \varphi \wedge \psi$. This is because the proofs $I_{1}, X_{2} \vdash_{\MQST} \Gamma \seq \Delta, \varphi$ and $I_{2}, X_{2} \vdash_{\MQST} \Gamma \seq \Delta, \psi$ extend to a proof $I_{1}, I_{2}, X_{2} \vdash_{\MQST} \Gamma \seq \Delta, \varphi \wedge \psi$, and moreover $X_{1} \vdash_{\MQST} I_{1}$ and $X_{1} \vdash_{\MQST} I_{2}$ imply that $X_{1} \vdash_{\MQST} I$.
\item Let $I$ be an interpolating set for $X_{1} \mid X_{2} \vdash_{\MQST} S$ for $S = \Gamma \seq \Delta, \varphi[x \mapsto y]$, where $y$ occurs free neither in $X_{1}, X_{2}$ nor in $\lvert \Gamma \rvert$ or $\lvert \Delta \rvert$. Let $I'$ be obtained from $I$ by universally quantifying with respect to all variables which do not occur free in $X_{1}$. Then $X_{1} \vdash_{\MQST} I$ and $X_{2}, I \vdash_{\MQST} S$ imply $X_{1} \vdash_{\MQST} I'$ and $X_{2}, I' \vdash_{\MQST} S$. Let $I''$ be obtained from $I'$ by existentially quantifying with respect to all variables which do not occur free in $X_{2}, S'$ for $S' = \Gamma \seq \Delta, \forall x\,\varphi$. In particular, $y$ does not occur free in $I'$, since it does not occur free in $X_{2}, S'$. The set of sequents $I''$ therefore only contains free variables shared by $X_{1}$ and $X_{2}, S'$. Moreover, $X_{1} \vdash_{\MQST} I''$ and $X_{2}, I'' \vdash_{\MQST} S$, using the fact that $y$ does not occur free in $I''$. Thus $I''$ is an interpolating set for $X_{1} \mid X_{2} \vdash_{\MQST} S'$.
\item Let $I$ be an interpolating set for $X_{1} \mid X_{2} \vdash_{\MQST} S$ for $S = \Gamma \seq \Delta, \varphi[x \mapsto t]$. Let $I'$ and $I''$ be obtained from $I$ in the same way as in the previous case. Then $X_{1} \vdash_{\MQST} I''$ as before, and moreover $X_{2}, I' \vdash_{\MQST} S$, so $X_{2}, I' \vdash_{\MQST} S'$ for $S' = \Gamma \seq \Delta, \exists x\,\varphi$ and $X_{2}, I'' \vdash_{\MQST} S'$. Thus $I''$ is an interpolating set for $X_{1} \mid X_{2} \vdash_{\MQST} S'$.
\end{enumerate}

  Now suppose that the last step of the derivation is an elimination. As in the proof of~\cite[Thm.~6.3.1]{TroelstraSchwichtenberg}, we define a \emph{main branch} of a normal derivation to be a branch of the derivation tree which goes from the root to one of the terminal nodes and in doing so only passes through the premises of introduction rules, the \emph{major} premises of elimination rules, or the premises of contraction rules. It follows that the terminal node of the branch is either an undischarged assumption of the proof or an instance of (GID).

  Each main branch forms a subsequence of a track, namely a subsequence where the parts between the major premise of a sidetrack rule and the minor premise have been removed. Since the last step is an elimination rule, it thus follows from Lemma~\ref{lemma: midsegment} that there are no introduction rules on a main branch. In particular, because there are no instances of either ($\land$R$\downarrow$) or ($\lor$L$\downarrow$) on a main branch, there is in fact exactly one main branch in the proof. Moreover, the terminal node of this main branch (possibly followed by a sequence of contractions) is a major premise of an elimination rule. This means that this terminal node cannot be an instance of (GID), otherwise it would form a cut segment.

Using the inductive hypothesis, we obtain some trivial cases for the unary invertible rules, plus twelve cases (depending on the division between $X_{1}$ and $X_{2}$) for the non-invertible rules.

  Let $I$ be an interpolating set for $S, X_{1} \mid X_{2} \vdash_{\MQST} T$. Let us deal with three of the six cases which arise in this situation, since the other three are entirely analogous:
\begin{enumerate}
\item Let $S = \Gamma \seq \Delta, \varphi$ and $S' = \Gamma \seq \Delta, \varphi \wedge \psi$. Then $I$ is also an interpolating set for $S', X_{1} \mid X_{2} \vdash_{\MQST} T$.
\item Let $S = \Gamma \seq \Delta, \varphi[x \mapsto t]$ and $S' = \Gamma \seq \Delta, \forall x\,\varphi$. Let $I'$ be obtained from $I$ by universally quantifying over all variables which do not occur in $S, X_{1}$. Then $S, X_{1} \vdash_{\MQST} I$, so $S', X_{1} \vdash_{\MQST} I$ and $S', X_{1} \vdash_{\MQST} I'$. Also, $X_{2}, I' \vdash_{\MQST} T$. Let $I''$ be obtained from $I'$ by existentially quantifying over all variables which do not occur in $X_{2}, T$. Then $X_{2}, I'' \vdash_{\MQST} T$, so $I''$ is an interpolating set for $S', X_{1} \mid X_{2} \vdash_{\MQST} T$.
\item Let $S = \Gamma \seq \Delta, \varphi[x \mapsto y]$ and $S' = \Gamma \seq \Delta, \exists x\,\varphi$. We may assume without loss of generality, by renaming $y$ if necessary, that $y$ occurs free neither in $S$ nor in $X_{1}, X_{2}, T$. Let $I'$ be obtained from $I$ by existentially quantifying over $y$, and let $I''$ be obtained from $I'$ by universally quantifying over all variables which do not occur in $S', X_{1}$. Then $S, X_{1} \vdash_{\MQST} I$ implies $S, X_{1} \vdash_{\MQST} I'$ and thus $S', X_{1} \vdash_{\MQST} I'$ and $S', X_{1} \vdash_{\MQST} I''$. Also, $X_{2}, I \vdash_{\MQST} T$ implies $X_{2}, I' \vdash_{\MQST} T$ and $X_{2}, I'' \vdash_{\MQST} T$. Let $I'''$ be obtained from $I''$ by existentially quantifying over all variables which do not occur in $X_{2}$ and $T$. Then $S, X_{1} \vdash_{\MQST} I'''$ and $X_{2}, I''' \vdash_{\MQST} T$, so $I'''$ is an interpolating set for $S', X_{1} \mid X_{2} \vdash_{\MQST} T$.
\end{enumerate}
  On the other hand, let $I$ be an interpolating set for $X_{1} \mid X_{2}, S \vdash_{\MQST} T$. Let us again only deal with three of the six cases which arise in this situation:
\begin{enumerate}
\item Let $S = \Gamma \seq \Delta, \varphi$ and $S' = \Gamma \seq \Delta, \varphi \wedge \psi$. Then $I$ is also an interpolating set for $X_{1} \mid X_{2}, S' \vdash_{\MQST} T$.
\item Let $S = \Gamma \seq \Delta, \varphi[x \mapsto t]$ and $S' = \Gamma \seq \Delta, \forall x\,\varphi$. Let $I'$ be obtained from $I$ by universally quantifying over all variables which do not occur in $X_{1}$. Then $X_{1} \vdash_{\MQST} I'$ and $X_{2}, S, I' \vdash_{\MQST} T$, so $X_{2}, S', I' \vdash_{\MQST} T$. Let $I''$ be obtained from $I'$ by existentially quantifying over all variables which do not occur in $X_{2}, S', T$. Then $X_{1} \vdash_{\MQST} I'$ and $X_{2}, S', I'' \vdash_{\MQST} T$, so $I''$ is an interpolating set for $X_{1} \mid X_{2}, S' \vdash_{\MQST} T$.
\item Let $S = \Gamma \seq \Delta, \varphi[x \mapsto y]$ and let $S' = \Gamma \seq \Delta, \exists x\,\varphi$. We may assume without loss of generality, by renaming $y$ if necessary, that $y$ occurs free neither in $S$ nor in $X_{1}, X_{2}, T$. Let $I'$ be obtained from $I$ by existentially quantifying over $y$ and let $I''$ be obtained from $I'$ by universally quantifying over all variables which do not occur in $X_{1}$. Then $X_{1} \vdash_{\MQST} I$ implies $X_{1} \vdash_{\MQST} I'$ and $X_{1} \vdash_{\MQST} I''$. Also, $X_{2}, S, I \vdash_{\MQST} T$ implies $X_{2}, S, I' \vdash_{\MQST} T$ and $X_{2}, S, I'' \vdash_{\MQST} T$. Let $I'''$ be obtained from $I''$ by existentially quantifying over all variables which do not occur in $X_{2}, S, T$. Then $X_{1} \vdash_{\MQST} I'''$ and $X_{2}, S, I''' \vdash_{\MQST} T$, so $I'''$ is an interpolating set for $X_{1} \mid X_{2}, S \vdash_{\MQST} T$.
\end{enumerate}
\end{proof}

\section{Conclusions and Open Problems}

The problem of endowing first-order Strict-Tolerant Logic with an adequate Gentzen-style proof theory is not the sort of problem that admits a unique solution, nor do we claim to have solved it for good. Be that as it may, we believe that there is something to be said for the systems we have introduced. On the one hand, they may help understand what proof-theoretic features a calculus should possess to mirror some natural semantic relations arising out of $\mathrm{ST}$-models. On a more general note -- and this holds true especially of $\MQST$, we think -- they may shed some light on possible ways to recover some of the deductive power of Cut in th context of sequent calculi that do not contain it, and whose operational rules are not all invertible. 

Being a preliminary foray into the topic, this paper leaves many issues unaddressed. At least the following problems remain open for future research:
\begin{itemize}
    \item The paper \cite{Prenosil} contains general split interpolation theorems for propositional sequent calculi with elimination rules similar to $\mathcal{ST}^P$, where either Cut or Identity are missing or suitably restricted. The problem as to whether these results extend to the first-order versions of such calculi is open at the time of writing. 
    \item \emph{Tolerant-Strict Logic} ($\mathrm{TS}$) is the logic dual to $\mathrm{ST}$ \cite{Pailos}. First-order $\mathrm{TS}$ faces a strong completeness issue similar to $\mathrm{ST}$, but addressing it may require non-trivial modifications of the toolbox deployed here. We think that this problem would deserve some consideration. 
    \item The relationships between $\mathcal{ST}^H$ and the Epsilon calculus, in the absence of Cut, have not been clarified. Shedding more light on this relation might help to understand the potential and limits of calculi with Henkin constants, above and beyond the specific proof system examined here.
\end{itemize}

\vspace{6mm}

\emph{Acknowledgments} This work was supported by PLEXUS (Grant Agreement no 101086295), a Marie Sklodowska-Curie action funded by the EU under the Horizon Europe Research and Innovation Programme. We also acknowledge the support of the Italian Ministry of University and Research, under the PRIN project DeKLA: Developing Kleene Logics and their Applications (2022SM4XC8), and of Fondazione di Sardegna, under the project Ubiquitous Quantum Reality (UQR): understanding the natural processes under the light of quantum-like structures (F73C22001360007). The second author's work was funded by the grant 2021 BP 00212 of the grant agency AGAUR of the Generalitat de Catalunya. We are extremely grateful to Elio La Rosa for providing numerous suggestions and a wealth of bibliographical material concerning the Epsilon calculus, and to Pierluigi Graziani, Rosalie Iemhoff and Dave Ripley for their comments. Finally, we thank two anonymous reviewers for their detailed and insightful feedback.

\end{document}